\documentclass[11pt,a4paper,reqno]{amsproc}
\usepackage{amsmath, amsthm, amscd, amsfonts, amssymb, graphicx, color}
\usepackage{subcaption}
\usepackage{enumerate}
\usepackage[bookmarksnumbered, colorlinks, plainpages]{hyperref}
\hypersetup{colorlinks=true,linkcolor=red, anchorcolor=green, citecolor=cyan, urlcolor=red, filecolor=magenta, pdftoolbar=true}

\newtheorem{theorem}{Theorem}[section]
\newtheorem{lemma}[theorem]{Lemma}
\newtheorem{proposition}[theorem]{Proposition}
\newtheorem{corollary}[theorem]{Corollary}
\theoremstyle{definition}
\newtheorem{definition}[theorem]{Definition}

\theoremstyle{remark}
\newtheorem{remark}[theorem]{Remark}

\numberwithin{equation}{section}

\usepackage[square,numbers]{natbib}

\begin{document}

\setcounter{page}{1}

\title[Characterization  of bi-parametric potential associated with the Dunkl Laplacian]
 {Characterization of bi-parametric potentials and rate of convergence of truncated hypersingular integrals in the Dunkl setting}

\author[Sandeep Kumar Verma and Athulya P ]{Sandeep Kumar Verma and Athulya P}

\address{Department of Mathematics, SRM University-AP, Amaravati, Guntur--522240, India}

\email{sandeep16.iitism@gmail.com, athulya.panoli97@gmail.com}



\begin{abstract}

In this work, we introduce the $\beta$-semigroup for $\beta > 0$, which unifies and extends the classical Poisson (for $\beta=1$) and heat (for $\beta=2$) semigroups within the Dunkl analysis framework. Leveraging this semigroup, we derive an explicit representation for the inverse of the Dunkl–Riesz potential and characterize the image of the function space $L_k^p(\mathbb{R}^n)$ for $1 \leq p < \frac{n + 2\gamma}{\alpha}$. We further define the bi-parametric potential of order $\alpha$ by
\[
\mathfrak{S}_k^{(\alpha,\beta)} = \left(I + (-\Delta_k)^{\frac{\beta}{2}}\right)^{-\frac{\alpha}{\beta}},
\]
and establish its inverse along with a detailed description of the associated range space. Our approach employs a wavelet-based method that represents the inverse as the limit of truncated hypersingular integrals parameterized by $\epsilon > 0$. To analyze the convergence of these approximations, we introduce the concept of $\eta$-smoothness at a point $x_0$ in the Dunkl setting. We show that if a function $f \in L_k^p(\mathbb{R}^n) \cap L_k^2(\mathbb{R}^n)$, for $1 \leq p \leq \infty$, possesses $\eta$-smoothness at $x_0$, then the truncated hypersingular approximations converge to $f(x_0)$ as $\epsilon \to 0^+$.
\end{abstract}
\maketitle

\section{Introduction}
The field of potential analysis in Euclidean spaces has a rich and well-established history \cite{Stein}, with classical developments centered around Riesz and Bessel potentials due to their fundamental roles in harmonic analysis and partial differential equations \cite{Aronszajn}. In this context, Flett \cite{Flett} introduced a novel fractional integro-differential operator, now known as the Flett potential, which offers improved localization properties and greater flexibility in addressing anisotropic structures and weighted function spaces. More recently, Aliev \cite{Aliev} proposed a broader framework known as the \emph{bi-parametric potential}, defined by the operator
\[
\left(I + (-\Delta)^{\frac{\beta}{2}}\right)^{-\frac{\alpha}{\beta}},
\]
where $I$ is the identity operator, $\Delta$ is the Euclidean Laplacian, and $\alpha, \beta > 0$ are parameters controlling the behavior of the potential. This formulation is particularly noteworthy, as it recovers the Bessel potential when $\beta = 2$ and the Flett potential when $\beta = 1$, thereby unifying a wide class of fractional potential operators under a single analytical framework.

A central problem in potential theory is to derive explicit inversion formulas for such potential-type integral operators. The theory of hypersingular integrals has gradually developed through the sustained investigation of this subject by numerous researchers (see \cite{Bagby,  Lizorkin, Stein}, etc). Rubin, in his seminal monograph
 \cite{Rubin-1986} established inversion formulas for the Riesz and Bessel potentials using the heat semigroup, via the finite difference scheme. This approach introduced a family of truncated hypersingular integrals indexed by a small parameter $\epsilon > 0$, generated by the heat semigroup.

Later, Aliev and Eryiğit \cite{Aliev-2013} investigated the convergence properties of these truncated families, showing that under the condition of $\eta$-smoothness at a point $x_0$, the truncated hypersingular integrals converge to the function value $f(x_0)$ as $\epsilon \to 0^+$. This line of analysis was further extended by Bayrakcı et al. \cite{Bayrakci}, who applied similar techniques to the Flett potential using the Poisson semigroup.

Beyond the finite difference method, Rubin developed an alternative \emph{wavelet-based approach} for inverting potential operators \cite{Rubin-1996, Rubin-1998}. In this method, the role of the finite difference operator is replaced by a wavelet measure, where the number of vanishing moments of the wavelet serves as an analog of the finite difference order. Building on this idea, Sezer and Aliev in 2010  proposed a novel characterization of Riesz potentials using beta-semigroups generated by fractional powers of the Laplacian \cite{Sezer}. These $\beta$-semigroups generalize both the heat semigroup (for $\beta = 2$) and the Poisson semigroup (for $\beta = 1$), with the parameter $\beta$ offering enhanced flexibility, such as the ability to reduce the number of vanishing moments required for convergence.

 In this work, we extend the aforementioned potential analysis within the framework of the Dunkl transform, a broad generalization of the classical Fourier and Hankel transforms  \cite{DJ1, D3}. The Dunkl theory was developed by Dunkl in the late 1980s and extensively studied by many authors (see \cite{DJ1, Rosler-98, Rosler-99, Rosler-2003}). The cornerstone of the Dunkl theory is the differential-difference operator \cite{D1}, which acts as a perturbation of the partial differentials in the Euclidean setting. The motivation for studying these operators stems from the theory of Riemannian symmetric spaces, where spherical functions can be expressed as multivariable special functions that depend on certain discrete sets of parameters. Dunkl operators play a crucial role in the analysis of special functions with reflection symmetries \cite{Rosler-2008}. Beyond their applications in symmetric spaces, Dunkl operators are also fundamental in mathematical physics \cite{Benchikha}, especially in the study of quantum many-body systems. Notably, they play a significant role in the analysis of the Calogero-Moser-Sutherland (CMS) model, which describes interacting particle systems with inverse-square potential interactions; see \cite{DV} and references therein.
 
An important advancement in Dunkl theory is the formulation of the Dunkl Laplacian \cite{D1}, an analogue of the classical Laplace operator which enables the extension of potential theory to broader and more generalized settings. In \cite{Xu}, the \emph{Dunkl-Riesz potential} \( \mathcal{I}_k^\alpha:= \left( -\Delta_k \right)^{-\alpha} \) has been introduced for \( 0 < \alpha < n + 2\gamma \), and its boundedness was investigated under the assumption that the reflection group is \( W = \mathbb{Z}_2^n \). Subsequently, Hassani et al.~\cite{Hassani} extended the theory to general finite reflection groups by expressing the Dunkl-Riesz potential in terms of the heat semigroup, thereby enabling a more unified and flexible analysis of its properties \cite{Gallardo}. These works laid the foundation for further study of fractional powers of the Dunkl Laplacian \cite{Ben-2025, Bouzeffour, Rejeb}.  
In the present paper, we adopt an alternative approach by representing the Dunkl-Riesz potential via the \emph{\( \beta \)-semigroup}, in the spirit of the construction developed in \cite{Sezer}. For $\beta>0$, the \( \beta \)-semigroup is defined as

$$\mathfrak{B}_k^{(\beta,t)}(f)(\xi)= \mathcal{W}_k^{(\beta,t)} \underset{k}{\ast} f(\xi),$$ 
where $\mathcal{W}_k^{(\beta,t)}(\xi)=\mathcal{D}_k^{-1}(e^{-t\|\cdot\|^{\beta}})(\xi)$ and $\underset{k}{\ast}$ represent the Dunkl convolution \cite{Thangavelu}. For $\beta\in (0,2)$, these semigroups were extensively studied by Rejeb in \cite{Rejeb}, particularly in the context of the fractional Dunkl Laplacian via Bochner’s subordination for strongly continuous semigroups. The fractional Laplacian of order $\alpha \in (0,1)$ can be treated as the inverse of the Riesz potential of order $2\alpha$. The $\beta$-semigroup generalizes both the Dunkl heat (for $\beta = 2$) and the Dunkl-Poisson semigroup (for $\beta = 1$) \cite{Rosler-98, Thangavelu}. Using the formulation of these semigroups, the Dunkl-Riesz potential can be represented  as follows: 

\begin{align*}
\mathcal{I}_k^\alpha (f)(x) =\frac{1}{c_k\Gamma(\alpha / \beta)} \int_0^{\infty} t^{\frac{\alpha}{\beta}-1} \mathfrak{B}_k^{(\beta, t)} (f)(x)\,dt.
\end{align*} This representation enables us to derive an inversion formula based on the wavelet method for general finite reflection groups. Furthermore, we establish a characterization of the image of $L_k^p(\mathbb{R}^n)$ for $1\leq p < \frac{n+2\gamma}{\alpha}$ under the  Dunkl-Riesz potential. Our main results can be summarized as follows:
\\

\noindent \textbf{Theorem A} (see Theorem \ref{Inversion-Riesz})
     Let $0< \alpha < n+2\gamma$, $\beta >0 $, $f\in L_k^p(\mathbb{R}^n)$ with $1\leq p < \frac{n+2\gamma}{\alpha}$ 
      and $\nu$ be a wavelet measure satisfying the conditions 
      \begin{align*}
          (i)\quad & \int_0^\infty s^{l}d|\nu|(s) < \infty \quad \text{for some }  l > \frac{\alpha}{\beta}, \\
          (ii)\quad  & \int_0^\infty s^m d\nu(s) = 0, \text{ 
 for} m = 0, 1, \ldots, \left\lfloor \frac{\alpha}{\beta} \right\rfloor,
\end{align*}
where \( \lfloor \cdot \rfloor \) denotes the greatest integer function.
     If $g=\mathcal{I}_k^\alpha (f),$ then 
     \begin{align*}
     \int_0^\infty  \mathbf{W}_k^{(\beta,t)}(g)(\xi) t^{-\frac{\alpha}{\beta}-1}dt = \lim_{\epsilon\longrightarrow  0}
       \mathfrak{T}_\epsilon (g)(\xi) =  C\left(\frac{\alpha}{\beta}, \nu\right)f(\xi),
     \end{align*} where 

\begin{align*}
       \mathbf{W}_k^{(\beta,t)} (f)(\xi) = \int_0^\infty  \mathfrak{B}_k^{(\beta, st)}(f)(\xi) d\nu(s),\quad 
   \mathfrak{T}_\epsilon (g)(\xi) = \int_\epsilon^\infty  \mathbf{W}_k^{(\beta,t)}(g)(\xi) t^{-\frac{\alpha}{\beta}-1}dt
\end{align*} and 
\[ 
C(r,\nu) = \int_0^{\infty} \frac{\mu(t)}{t^{1+r}}dt =  
\begin{cases} 
  \Gamma(-r)\int_0^{\infty}s^{r}d\nu(s) & \text{ if } r \notin \mathbb{N}_0 \\ 
  \frac{(-1)^{r+1}}{r !} \int_0^{\infty}s^{r} \ln{s}d\nu(s) & \text{ if } r \in \mathbb{N}_0. \\  
\end{cases} 
\] 
The function \( \mu(t) = \int_0^{\infty} e^{-ts} \, d\nu(s) \) represents the Laplace transform of the wavelet measure \( \nu \). The convergence of the associated integral expression is considered in the  \( L_k^p(\mathbb{R}^n) \)-norm. Moreover, if \( f \in \mathcal{C}_0(\mathbb{R}^n) \cap L_k^p(\mathbb{R}^n) \), the convergence is uniform on \( \mathbb{R}^n \).\\

\noindent \textbf{Theorem B} (see Theorem \ref{Char.R.P.Space})
    Let $0<\alpha<n+2\gamma, \quad 1<p<\frac{n+2\gamma}{\alpha}$, $\beta > \alpha$, and $\nu$ be the wavelet measure with $C(\frac{\alpha}{\beta},\nu)\neq 0.$
    Then $g \in \mathcal{I}_k^\alpha(L_k^p(\mathbb{R}^n))$ if and only if $g \in L_k^q(\mathbb{R}^n)$ for  $q= \frac{p(n+2\gamma)}{n+2\gamma-p\alpha}$ and 
    $\underset{\epsilon>0}{\sup} \| \mathfrak{T}_\epsilon g\|_{L_k^p(\mathbb{R}^n)} < \infty.$\\

For \( \alpha > 0 \), the Dunkl Bessel and Dunkl Flett potentials are defined by \cite{Ben}
\[
\mathcal{J}_k^\alpha = \left(I - \Delta_k\right)^{-\alpha/2}
\quad \text{and} \quad
\mathcal{F}_k^\alpha = \left(I + (-\Delta_k)^{1/2} \right)^{-\alpha},
\]
respectively. These two potentials exhibit a structural similarity, with their primary distinction arising from the presence of a fractional power of the Dunkl Laplacian. Motivated by this observation, and inspired by the work of Aliev in the Euclidean setting \cite{Aliev}, we introduce a more general class of potential operators associated with the Dunkl Laplacian. Specifically, we define the \emph{bi-parametric potential operator} as
\[
\mathfrak{S}_k^{(\alpha, \beta)} := \left( I + \left( -\Delta_k \right)^{\beta/2} \right)^{-\alpha/\beta},
\]
which generalizes the Dunkl Bessel and Flett potentials for different choices of \( \beta \).
We establish the following integral representation of \( \mathfrak{S}_k^{(\alpha, \beta)} \):
\[
\mathfrak{S}_k^{(\alpha, \beta)} (f)(\xi) = \frac{1}{\Gamma(\alpha/\beta)} \int_0^\infty t^{\frac{\alpha}{\beta}-1} e^{-t} \, \mathfrak{B}_k^{(\beta, t)} (f)(\xi) \, dt,
\]
where \( \mathfrak{B}_k^{(\beta, t)} \) denotes the \(\beta\)-semigroup associated with the Dunkl Laplacian. Furthermore, we derive an inversion formula based on a wavelet-type transform and provide a characterization of the associated potential spaces. \\

\noindent\textbf{Theorem C} (see Theorem \ref{Thrm-Inversion of beta}) Let \( \alpha, \beta > 0 \), \( f \in L_k^p(\mathbb{R}^n) \) for \( 1 \leq p \leq \infty \), and let \( \nu \) be a wavelet measure satisfying the conditions:
\begin{align*}
\text{(i)} & \quad \int_0^\infty s^{l} \, d|\nu|(s) < \infty \quad \text{for some } l > \frac{\alpha}{\beta}, \\
\text{(ii)} & \quad \int_0^\infty s^m \, d\nu(s) = 0 \quad \text{for all } m = 0, 1, \ldots, \left\lfloor \frac{\alpha}{\beta} \right\rfloor,
\end{align*}
where \( \lfloor \cdot \rfloor \) denotes the greatest integer function. Then the following inversion formula holds:
\[
\int_0^\infty s^{-\frac{\alpha}{\beta}-1} \, \mathcal{V}_k^{(\beta,s)}\left( \mathfrak{S}_k^{(\alpha, \beta)}(f) \right)(x) \, ds 
= \lim_{\epsilon \to 0} \mathcal{V}_\epsilon\left( \mathfrak{S}_k^{(\alpha, \beta)}(f) \right)(x) 
= C\left(\frac{\alpha}{\beta}, \nu\right) f(x),
\]
where
\[
\mathcal{V}_\epsilon\left( \mathfrak{S}_k^{(\alpha, \beta)}(f) \right)(x) := \int_\epsilon^\infty s^{-\frac{\alpha}{\beta}-1} \, \mathcal{V}_k^{(\beta,s)}\left( \mathfrak{S}_k^{(\alpha, \beta)}(f) \right)(x) \, ds,
\]
and the constant \( C\left( \frac{\alpha}{\beta}, \nu \right) \) is as defined in Theorem~\ref{Inversion-Riesz}. The convergence of the limit is to be understood in the \( L_k^p(\mathbb{R}^n) \)-norm for \( 1 \leq p < \infty \), and in the uniform norm when \( p = \infty \), provided \( f \in \mathcal{C}_0(\mathbb{R}^n) \).\\

\textbf{Theorem D} (see Theorem \ref{p.Inversion})  Let $\alpha, \beta>0$ and $f\in L_k^p(\mathbb{R}^n)$ for $1<p<\infty$. Then $f\in \mathfrak{S}_k^{(\alpha, \beta)}(L_k^p) $ if and only if $$\underset{\epsilon>0}{\sup}\|\mathcal{V}_\epsilon (f)\|_{L_k^p(\mathbb{R}^n)} < \infty .$$ \\

 In Theorems~\ref{Inversion-Riesz} and \ref{Thrm-Inversion of beta}, the inversions of the Dunkl-Riesz potential and the bi-parametric potential are expressed in terms of truncated hypersingular integrals, namely \( \mathfrak{T}_\epsilon(\mathcal{I}_k^{\alpha}(f)) \) and \( \mathcal{V}_\epsilon\left( \mathfrak{S}_k^{(\alpha, \beta)}(f) \right) \), respectively. This formulation naturally raises the question whether an analogue of \cite[Theorem 3.1]{Aliev-2013} can be established within the Dunkl framework, specifically for Riesz and bi-parametric potentials constructed via the \(\beta\)-semigroup.

In the classical setting, approximation results for Riesz, Bessel, and Flett potentials have been developed using semigroups such as the heat, Poisson, and metaharmonic semigroups \cite{Aliev-2014, Aliev-2013, Bayrakci, Eryigit}, all of which possess kernels with positive values. In contrast, within the Dunkl framework, the positivity of the \(\beta\)-semigroup kernel is known only for \( 0 < \beta \leq 2 \) (see Proposition~\ref{Proposition-main}~(ii)), posing additional analytical challenges.

Moreover, the Dunkl setting exhibits several unresolved issues, notably the lack of a comprehensive understanding of the boundedness of the Dunkl translation operator \( \tau_y \) on \( L_k^p(\mathbb{R}^n) \) for \(p\neq 2 \). Since this operator is fundamental in the definition of \(\eta\)-smoothness, our analysis is necessarily restricted to the Hilbert space \( L_k^2(\mathbb{R}^n) \).

We primarily focus on extending the existing theory by introducing suitable modifications that enable the study of the convergence behavior of truncated hypersingular integrals within the framework of the wavelet method. To the best of our knowledge, the rate of convergence of such integrals has only been examined in the context of potential theory via the finite difference method.

The results presented in Section~\ref{S:5} establish explicit convergence rates for the Dunkl-Riesz and bi-parametric potentials formulated through the \(\beta\)-semigroup. Furthermore, one can validate that these results are consistent with classical cases when \(\beta = 1\), \(\beta = 2\), and \( k = 0 \), aligning with the findings in \cite{Aliev-2014, Aliev-2013}. This can be formulated as
\\
\\
\noindent \textbf{Theorem E} (see Theorem \ref{Riez-R-of-C}) 
    Let \( f \in L_k^p(\mathbb{R}^n) \cap L_k^2(\mathbb{R}^n) \) for \( 1 \leq p < \infty \), and suppose that \( f \) possesses \(\eta\)-smoothness at the point \( x_0 \in \mathbb{R}^n \). Then the following estimate holds:
\begin{align*}
    \left| \mathfrak{T}_\epsilon(\mathcal{I}_k^{\alpha}(f))(x_0)- f(x_0) \right| \leq C\,\eta(Y(\epsilon)) \quad \text{as } \epsilon \to 0^+,
\end{align*}
where \( 0 < \alpha < \frac{n+2\gamma}{p} \), \( C>0\) is a constant independent of \( \epsilon\), and the function \( Y(\epsilon) \) is defined by
\begin{equation*}
     Y(\epsilon) =   
\begin{cases} 
  \epsilon^{1/\beta}, & \text{if } \beta \geq 1, \\ 
  \epsilon^{\beta},   & \text{if } 0 < \beta \leq 1.
\end{cases}
\end{equation*}
\\
\noindent \textbf{Theorem F} (see Theorem \ref{Rate-Pot})
  Let $f \in L_k^p(\mathbb{R}^n)\cap L_k^2(\mathbb{R}^n)$ for $1\leq p \leq \infty$, and suppose that $f$ exhibits $\eta$-smoothness at a point $x_0$. Then the following point-wise estimate holds:
  \begin{align*}
      \left| \mathcal{V}_\epsilon\left( \mathfrak{S}_k^{(\alpha, \beta)}(f) \right)(x_0)-f(x_0) \right| \leq C\eta(Y(\epsilon)), \quad \text{ as $\epsilon \longrightarrow 0^+$}, 
  \end{align*}
  where $\alpha, \beta >0$, $C>0$ is a constant independent of $\epsilon$, and the function \(Y(\epsilon)\) is as defined in Theorem \ref{Riez-R-of-C}.

The paper is organized as follows: Section \ref{S:2} provides a brief overview of the Dunkl operator, Dunkl transform, and the potentials associated with the Dunkl Laplacian. Building upon this foundation, Section \ref{S:3} focuses on the construction of the $\beta$-semigroup and discusses its properties. Leveraging this semigroup, we derive an inversion formula for the Dunkl-Riesz potential and subsequently characterize the associated Dunkl-Riesz potential spaces. In Section \ref{S:4}, we introduce the concept of bi-parametric potentials. Furthermore, we derive the corresponding inversion formula and provide a characterization of the bi-parametric potential spaces. Section \ref{S:5} deals with the rate of convergence of truncated functions associated with the inversion of Riesz and bi-parametric potentials.

\section{Preliminaries} \label{S:2}
This section compiles some background on the Dunkl theory. More details can be found in \cite{DJ1, D1, D2, D3, Rosler-2003} and the references therein. 

\subsection{Dunkl operator}  
We consider the Euclidean space $\mathbb{R}^n$, endowed with the standard inner product $\langle x, y \rangle = \sum_{j=1}^n x(j)y(j)$. A root system $\mathcal{R}$ is a finite set of non-zero vectors in $\mathbb{R}^n$ satisfying the properties $ \mathcal{R} \cap \mathbb{R}u = {\pm u} $ and $ \sigma_u(\mathcal{R}) = \mathcal{R} $ for every $u \in \mathcal{R}$, where $\sigma_u$ denotes the reflection with respect to the hyperplane orthogonal to $u$. For simplicity, we shall consider the normalized root systems, i.e., $|u|^2 = 2$ for all $u \in \mathcal{R}$. The root system $\mathcal{R}$ can be partitioned into two disjoint subsets, the positive roots $\mathcal{R}_+$ and the negative roots $\mathcal{R}_{-}$, via any hyperplane passing through the origin.
The group $W$ generated by the reflections $\sigma_u$ is called the Weyl group (reflection group) of the root system.
A multiplicity function is a $W$-invariant function $k: \mathcal{R} \to \mathbb{C}$.
For further background on root systems and reflection groups, we refer the reader to \cite{GB, Hu, RKAN}.
 
 Given a positive root system $\mathcal{R}_+ $ and a multiplicity function $k$, the associated Dunkl operator was introduced by Dunkl in \cite{D1} and is defined as follows:
\begin{eqnarray*} 
    \mathcal{T}_{\xi}(f)(x) = \partial_{\xi}(f)(x)+ \sum_{u\in \mathcal{R}_+}k(u) \langle u, \xi \rangle \frac{f(x)-f(\sigma_u(x))}{\langle u, x\rangle}, \quad f\in \mathcal{C}^{\prime}(\mathbb{R}^n).
\end{eqnarray*} 

When the multiplicity function $k$ is identically zero, the Dunkl operator coincides with the standard directional derivative in the direction of $\xi$, thereby illustrating its interpretation as a deformation of the classical derivative. A counterpart to the classical exponential function in this setting is the Dunkl kernel $E_k$ \cite{D2}. The kernel $ E_k(\cdot,y)$ is uniquely characterized as the solution to the eigenvalue problem $\mathcal{T}_{\xi}(f)=\langle \xi, y \rangle f$ for all $\xi \in \mathbb{C}^n$ subject to the initial condition $f(0)=1$ cf. \cite{ Opdam}. The Dunkl kernel $E_k$ is fundamental to the theory, as it paves the way for the formulation of an integral transform analogous to the Fourier transform, known as the Dunkl transform $\mathcal{D}_k$ \cite{D3}.  We denote $\mathcal{T}_{j}$ for $\mathcal{T}_{e_j}$, where $\{e_1, e_2, \cdots e_n \}$ is the standard orthonormal basis for $\mathbb{R}^n$. The Dunkl Laplacian $\Delta_k$ \cite{D1}, defined by $\Delta_k= \sum_{j=1}^n \mathcal{T}_j^2$,  can also be expressed in terms of the Euclidean Laplacian  $\Delta$ and gradient $\nabla$ as follows:
\begin{eqnarray*}
    \Delta_k(f)(x) = \Delta (f)(x)+2 \sum _{u \in \mathcal{R}_+} k(u)\left( 
\frac{\langle \nabla f(x),u\rangle }{\langle u,x\rangle} - \frac{f(x)-f(\sigma_u(x))}{\langle u,x\rangle^2}\right), \quad f\in \mathcal{C}^2(\mathbb{R}^n). 
\end{eqnarray*}

\subsection{Dunkl transformation}
Let $k$ be a non-negative multiplicity function, which will be fixed throughout this paper. The weight function $w_k$ is defined as
\begin{align*}
    w_k(x) = \prod_{u\in \mathcal{R}_+}|\langle u,x\rangle|^{2k(u)}.
\end{align*}
The weight function is invariant under the group action and homogeneous of degree $2\gamma$, where  $\gamma  = \sum_{u\in \mathcal{R}_+} k(u)$.\\
For $1\leq p <\infty,$ we define the weighted Lebesgue space  as:
\begin{align*}
    L_k^p(\mathbb{R}^n) = \{ f:\mathbb{R}^n \longrightarrow \mathbb{C} \text{ measurable function and } \int_{\mathbb{R}^n}|f(x)|^pw_k(x)dx< \infty
    \},
\end{align*} where $dx$ is the Lebesgue measure on $\mathbb{R}^n,$ and for $p=\infty$,
\begin{align*}
    L^{\infty}_k(\mathbb{R}^n) = \{ f:\mathbb{R}^n \longrightarrow \mathbb{C} \text{ measurable function and } \text{ ess.sup}|f(x)| < \infty \}.
\end{align*}
A function $f:\mathbb{R}^n \longrightarrow \mathbb{C}$ is said to be a radial function if there exists a function $F_0$ on the non-negative real line such that $f(x)=F_0(\|x\|)$ for all $x \in \mathbb{R}^n.$ The collection of all radial functions on $L_{k}^p(\mathbb{R}^n)$ is denoted by $L_{k,\text{rad}}^p(\mathbb{R}^n)$.\\
For  $ f \in L_k^1(\mathbb{R}^n)$, the Dunkl transform \cite{DJ1,D3} is defined as
\begin{align*}
    \mathcal{D}_k(f)(x) =  c_k \int_{\mathbb{R}^n}f(y)E_k(-ix,y)w_k(y)dy,
\end{align*}
where $ c_k ^{-1} = \int_{\mathbb{R}^n} e^{-{|x|^2}/{2}}w_k(x)dx$. The inverse Dunkl transform is given by the relation $\mathcal{D}_k^{-1}(f)(x)=\mathcal{D}_k(f)(-x)$. Once the multiplicity function becomes zero, the transform $\mathcal{D}_k$  reduces to the Fourier transform and satisfies the following properties: 
\begin{enumerate}[$(i)$]
    \item $\mathcal{D}_k$ is a topological automorphism on Schwartz space $\mathcal{S}(\mathbb{R}^n)$.
    \item $\mathcal{D}_k$ extends to an isometric isomorphism of $L_k^2(\mathbb{R}^n)$.
    \item Riemann–Lebesgue lemma: For all $f\in L_k^1(\mathbb{R}^n)$, the Dunkl transform $\mathcal{D}_k(f) \in \mathcal{C}_0(\mathbb{R}^n)$.
\end{enumerate}

\begin{definition}\textbf{Dunkl translation:}\cite{Thangavelu}
    For a fixed $y\in \mathbb{R}^n,$ the Dunkl translation operator $\tau_y:L_k^2(\mathbb{R}^n) \longrightarrow L_k^2(\mathbb{R}^n)$ is defined by the integral
     \begin{align*}
        \tau_yf(x)= \int_{\mathbb{R}^n} E_k(
        ix,\xi)E_k(iy,\xi)\mathcal{D}_k(f)(\xi)w_k(\xi)d\xi.
     \end{align*}
    \end{definition}
The above integral definition remains valid for the function class $ \mathcal{A}_k(\mathbb{R}^n) = \{ f \in L_k^1(\mathbb{R}^n):\,\, \mathcal{D}_k(f) \in L^1_k(\mathbb{R}^n)\}$.\\

The Dunkl translation possesses the following properties \cite{Rosler-2003, trimeche-2002}.
\begin{proposition} \label{translation property}
    \begin{enumerate}[(i)]
     \item $\mathcal{S}(\mathbb{R}^n)$ is invariant under the Dunkl translation.
    \item Let $f\in \mathcal{S}(\mathbb{R}^n).$ Then for any $x,y\in \mathbb{R}^n$ we have  
    \begin{eqnarray*}
        \tau_yf(x)=\tau_xf(y). 
    \end{eqnarray*}
   \item Let $f \in L^1_{k,\text{rad}}(\mathbb{R}^n)$. Then we have 
    \begin{align*} 
        \int_{\mathbb{R}^n} \tau_yf(x)w_k(x)dx = \int_{\mathbb{R}^n} f(x) w_k(x)dx.
    \end{align*} 
\end{enumerate}
\end{proposition}

The following theorem states the boundedness of the translation operator.
 \begin{theorem}\cite{Thangavelu}\label{C.Trns}
    For $1\leq p\leq 2,$ the Dunkl translation operator $\tau_y:L_{k,\text{rad}}^p(\mathbb{R}^n) \longrightarrow L_{k}^p(\mathbb{R}^n) $ is a bounded map.  
  \end{theorem}
\begin{definition}
    \textbf{Dunkl convolution:}  Let $f,g \in L_k^2(\mathbb{R}^n)$. The Dunkl convolution product  of $f$ and $g$ is denoted by $f\underset{k}{\ast} g$ and defined  as \cite{Thangavelu}
       \begin{align*}
          f \underset{k}{\ast} g(x) = \int_{\mathbb{R}^n} f(y)\tau_xg(-y)w_k(y)dy.
       \end{align*}
\end{definition}
Now, we list some results of the Dunkl convolution operator \cite{Thangavelu}. 
\begin{proposition} \label{D.C-Property}
   Let $f,g\in L_k^2(\mathbb{R}^n)$. Then the following holds:
   \begin{enumerate}[(i)]
       \item $(f\underset{k}{\ast}g)(x)=(g\underset{k}{\ast}f)(x)$ 
       \item $\mathcal{D}_k(f\underset{k}{\ast}g)(x)=\mathcal{D}_k(f)(x)\mathcal{D}_k(g)(x).$
   \end{enumerate}
\end{proposition}
\begin{theorem} \label{D.C.continuity}
    For a bounded radial function $g$ in $L_k^1(\mathbb{R}^n)$, the convolution operator is a bounded operator on $L_k^p(\mathbb{R}^n)$, where $1\leq p\leq\infty$. In particular,
    \begin{align*}
        \| f\underset{k}{\ast}g  \|_{L_k^p(\mathbb{R}^n)} \leq \|g||_{L_k^1(\mathbb{R}^n)} \|f\|_{L_k^p(\mathbb{R}^n)}.
    \end{align*}
\end{theorem} 

\begin{definition}\textbf{Dunkl maximal function:} Let $f\in L_k^2(\mathbb{R}^n)$. The maximal function $\mathcal{M}_k$ associated with Dunkl transform is defined by \cite{Thangavelu} 
\begin{align*}
    \mathcal{M}_k(f)(x) = \underset{r>0}{\sup} \frac{1}{d_k r^{n+2\gamma}}\left|f\underset{k}{\ast}\chi_{B(0,r)} \right|, 
\end{align*} where $\chi_{B(0,r)}$ is the characteristic function of the ball of radius r and centered at the origin and $d_k=\int_{\mathbb{S}^{n-1}} w_k(\xi^\prime)d\sigma(\xi^\prime)$. Here $d\sigma$ denotes the Lebesgue measure on the unit sphere.
    
\end{definition}

\subsection{Riesz, Bessel, and Flett potentials associated with Dunkl Laplacian}
We review the Riesz, Bessel, and Flett potentials associated with the Dunkl transform \cite{Xu, Ben}. Let us first recall certain semigroups, such as heat and Poisson semigroups, associated with the Dunkl transform. For simplicity, we use the terms $k$-heat and $k$-Poisson semigroup instead of the heat and Poisson semigroup associated with the Dunkl transform. For more details related to these semigroups, the readers are referred to the works of R\"osler \cite{Rosler-98}, Thangavelu and Xu \cite{Thangavelu}, and Ben Salem et. al \cite{Ben}.\\ Let $f\in \mathcal{S}(\mathbb{R}^n)$. The k-heat semigroup is defined as:
\begin{align*}
    H_k^t(f)(x)= f\underset{k}{\ast} h_k^t (x),
\end{align*}where $h_k^t(x)= (2t)^{-(\gamma+\frac{n}{2})}e^{-\frac{\|x\|^2}{4t}}$.  The $k$-Poisson semigroup is defined as:
\begin{align*}
    \mathcal{P}_k^t(f)(x) = f\underset{k}{\ast} p_k^t (x),
\end{align*} where $p_k^t(x) = C_{n,k} \frac{t}{(t^2+\|x\|^2)^{2\gamma+\frac{n+1}{2}}}$ with  $C_{n,k}= \frac{2^{\gamma+\frac{n}{2}}}{\Gamma(1/2)} \Gamma(2\gamma+\frac{n+1}{2}).$ Also note that $ h_k^t(x)= \mathcal{D}_k\left( e^{-t\|\cdot\|^2} \right)(x) \quad \text{ and } \quad p_k^t(x)= \mathcal{D}_k \left(e^{-t\|\cdot\|}\right)(x). $
\\

 For $0< \alpha < 2\gamma+n$. The Dunkl-Riesz potential of order $\alpha$  \cite{Xu} is defined as  
\begin{align*}
    \mathcal{I}_k^{\alpha}(f)(x)= \frac{1}{d_k^{\alpha}}\int_{\mathbb{R}^n}\tau_{\xi}f(x) \frac{1}{\|\xi\|^{2\gamma+n-\alpha}} w_k(\xi)d\xi,
\end{align*} where $d_k^{\alpha}= 2^{-\gamma-\frac{n}{2}+\alpha} \frac{\Gamma(\alpha/2)}{\Gamma(\gamma+(n-\alpha)/2)}$.  \\
The Dunkl-Bessel potential of order $\alpha >0$  is given by the convolution product \cite{Ben} 
\begin{align*}
    \mathcal{J}_k^{\alpha}(f) = f \underset{k}{\ast} G_k^{\alpha} \text{ where } G_k^{\alpha}(x)=\mathcal{D}_k \left((1+\|\cdot\|^2)^{-\alpha/2}\right)(x).
\end{align*} This can also be expressed in terms of the $k$-heat semigroup as follows:
\begin{align*}
    \mathcal{J}_k^{\alpha}(f)(x)= \frac{1}{\Gamma(\alpha/2)} \int_0^{\infty} t^{\frac{\alpha}{2}-1} e^{-t} H_k^t(f)(x) dt.
\end{align*} Similarly,  the Dunkl-Flett potential of order $\alpha >0$  is given by 
\begin{align*}
    \mathcal{F}_k^{\alpha}(f) = f \underset{k}{\ast} U_k^{\alpha}, \text{ where } U_k^{\alpha}(x)=\mathcal{D}_k\left((1+\|\cdot\|)^{-\alpha}\right)(x).
\end{align*} It can also be written in the form of $k$-Poisson semigroup as follows,
\begin{align*}
    \mathcal{F}_k^{\alpha}(f)(x)= \frac{1}{\Gamma(\alpha)} \int_0^{\infty} t^{\alpha-1} e^{-t} \mathcal{P}_k^t(f)(x) dt.
\end{align*}

\section{Characterization of Dunkl-Riesz potential} \label{S:3}
Fractional integral operators are powerful mathematical tools in various fields, extending classical integration to non-integral orders. Their applications span pure mathematics, physics, engineering, and data science. A fundamental example of such operators is the Riesz potential, which plays a crucial role in harmonic analysis and potential theory \cite{Rubin-1996, Rubin-1998}. The Dunkl–Riesz potential has been studied extensively by various researchers; for reference, see \cite{Gallardo, Gorbachev-2021, Hassani}. 
In this section, we explore the Riesz potential through the framework of the $\beta$-semigroup, establish its inverse using the wavelet method, and characterize the Dunkl-Riesz potential space. 
\\

\noindent We recall the representation of Dunkl-Riesz potential in terms of $k$-heat semigroup \cite{Hassani}.

\begin{lemma} \label{Riesz-Heat}
    Let $0<\alpha< 2\gamma+n$ and $f\in \mathcal{S}(\mathbb{R}^n).$ Then Dunkl-Riesz potential admits the following integral representation
\begin{align} \label{Riesz-1}
    \mathcal{I}_k^{\alpha}(f)(x)= \frac{1}{c_k\Gamma(\alpha/2)} \int_0^{\infty} t^{\frac{\alpha}{2}-1} H_k^t(f)(x) dt.
\end{align}
\end{lemma}

\noindent In the same line, we can derive the representation of Dunkl-Riesz potential in conjunction with the $k$-Poisson semigroup.  
\begin{lemma}\label{Riesz-Poisson}
    For $0< \alpha< 2\gamma+n$ and $f \in \mathcal{S}(\mathbb{R}^n),$ we have the following representation for Dunkl-Riesz potential,
    \begin{align} \label{Riesz-2}
        \mathcal{I}_k^{\alpha}(f)(x)= \frac{1}{c_k\Gamma(\alpha)} \int_0^{\infty} t^{\alpha-1} \mathcal{P}_k^t(f)(x) dt.
    \end{align}
\end{lemma}
\begin{proof}
    The integral representation of the $k$-Poisson semigroup and Fubini's theorem will give the identity
    \begin{align*}
        \int_0^{\infty} t^{\alpha-1}\mathcal{P}_k^\alpha (f)(x) dt &= C_{n,k}\int_{\mathbb{R}^n}\tau_x f(\xi) \int_0^{\infty} \frac{t^\alpha}{(t^2+\|\xi\|^2)^{\gamma+\frac{n+1}{2}}} dt w_k(\xi)d\xi.
    \end{align*}Substituting $u=\frac{t}{\|\xi\|}$ followed by $u=\tan(\theta),$ we have
    \begin{align*}
         \int_0^{\infty} t^{\alpha-1}\mathcal{P}_k^\alpha (f)(x) dt &= \frac{C_{n,k}}{2} B\left(\frac{\alpha+1}{2}, \gamma+\frac{n-\alpha}{2}\right) d_{k}^\alpha \mathcal{I}_k^\alpha(f)(x), 
    \end{align*}  where $B(\alpha, \gamma)$ is the beta function.
    The identity will follow from the properties of the beta function.
\end{proof} 
\subsection{$\beta$-semigroup} 
Analogous to the classical approaches in the Euclidean framework \cite{Sezer}, we represent the Riesz potential in terms of the  $\beta$-semigroup.  This semigroup serves as a generalization of the $k$-Poisson and $k$-heat semigroups. The $\beta$-semigroups arise in various contexts; for instance, characterizing the fractional Dunkl–Laplacian as the infinitesimal generator of the fractional Dunkl heat semigroup \cite{Rejeb}. In our analysis, we focus on studying these semigroups in the context of potential operators.

To unify the two representations of the Dunkl-Riesz potential in  \eqref{Riesz-1} and \eqref{Riesz-2}, we introduce a new parameter $\beta>0$, resulting the following general expression 
\begin{align}\label{Riesz-beta}
\mathcal{I}_k^\alpha (f)(x) = \frac{1}{c_k\Gamma(\alpha / \beta)} \int_0^{\infty} t^{\frac{\alpha}{\beta}-1} \mathfrak{B}_k^{(\beta, t)} (f)(x)\,dt,
\end{align}
where 
\begin{align} \label{beta-semigroup}
\mathfrak{B}_k^{(\beta, t)} (f)(x) = c_k \int_{\mathbb{R}^n} e^{-t\|\xi\|^\beta} \mathcal{D}_k (f)(\xi) E_k(ix, \xi) w_k(\xi)\,d\xi.
\end{align} 

The integral representation \eqref{beta-semigroup} allows a convolution representation for $\mathfrak{B}_k^{(\beta,t)}(f)$  as follows
\begin{align} \label{beta-convolution}
    \mathfrak{B}_k^{(\beta,t)}(f)(x)= \mathcal{W}_k^{(\beta,t)} \underset{k}{\ast} f(x),
\end{align}
 where $\mathcal{W}_k^{(\beta,t)}(\xi)=\mathcal{D}_k^{-1}(e^{-t\|\cdot\|^{\beta}})(\xi)$. The kernel function $\mathcal{W}_k^{(\beta,t)}(\xi)$ is well-defined because, for all $\beta, t>0$ the function $e^{-t\|\cdot\|^{\beta}}\in L_k^1(\mathbb{R}^n)$. Indeed, the spherical-polar coordinates $\xi=r\xi^\prime,$ where $\xi^\prime\in \mathbb{S}^{n-1},$ we have
 \begin{align*}
     \int_{\mathbb{R}^n} e^{-t\|\xi\|^{\beta}} w_k(\xi)d\xi =  \int_0^{\infty} \int_{\mathbb{S}^{n-1}} e^{-tr^\beta}w_k(\xi^\prime)d\sigma(\xi^\prime) r^{2\gamma+n-1}dr= 
     \frac{ d_k \beta }{t^{\beta(2\gamma+n)}} \Gamma\left(\beta(2\gamma+n)\right).
  \end{align*}  
 
\noindent Now, we are in a position to list the properties of the function $\mathcal{W}_k^{(\beta,t)}$ and the convolution operator $\mathfrak{B}_k^{(\beta,t)}$. As a special case of $\beta\in (0,2)$, a few properties can be found in \cite{Rejeb}.

\begin{proposition} \label{Proposition-main}
      Suppose $ 0<\beta <\infty, \, t>0$ and $\xi \in \mathbb{R}^n$. 
      \begin{enumerate}[$(i)$]
          \item For any $s>0$, 
          \begin{align*}
              \mathcal{W}_k^{(\beta, s t)}(s^{1/\beta}\xi) = s^{-(n+2\gamma)/\beta}\mathcal{W}_k^{(\beta, t)}(\xi).  
          \end{align*} 
          \item  For $0<\beta \le 2, \quad \mathcal{W}_k^{(\beta, t)}$ is non-negative function .\\
          \item For  even values of $\beta$, $\mathcal{W}_k^{(\beta, t)}(\xi)$ is rapidly decreasing as $\|\xi\|\rightarrow \infty.$ Moreover, for any $\beta>0$ and $t>0$, $\mathcal{W}_k^{(\beta, t)}\in \mathcal{C}_0(\mathbb{R}^n)$. In particular, 
          \begin{align} \label{Pro-Main-3}
              \lim_{\|\xi\| \rightarrow \infty} \|\xi\|^{2\gamma+n+\beta-1}\left|\mathcal{W}_k^{(\beta, t)}(\xi)\right|= t2^{\gamma+\frac{n}{2}+\beta-2} \Gamma \left( \frac{2\gamma+n+\beta-1}{2} \right) \Gamma\left( \frac{\beta+1}{2} \right).
          \end{align}
          \item For $0<\beta <\infty$ and $t>0,$ 
          \begin{align*}
               c_k\int_{\mathbb{R}^n} \mathcal{W}_k^{(\beta, t)}(\xi)w_k(\xi)d\xi =1.
          \end{align*}
          \item For $1\le p \le \infty, \, f \in L_k^p(\mathbb{R}^n)$ and for all $t>0,$ we have
          \begin{align}
              \| \mathfrak{B}_k^{(\beta,t)}(f) \|_{L_k^p(\mathbb{R}^n)} &\le C({\beta}) \|f\|_{L_k^p(\mathbb{R}^n)}, \label{prop-main-5-0}\\
          \text{where } C({\beta})= c_k \int_{\mathbb{R}^n} |\mathcal{W}_k^{(\beta,1)}(\xi)|w_k(\xi)d\xi <\infty. & \text{ If } 0< \beta \le 2, \text{ then } C({\beta})=1.
        \label{prop-main-5}  \end{align}
          \item Let $ f \in L_k^p(\mathbb{R}^n)$  for $p \in [1,\infty) $  . Then  
          \begin{align} \label{Pro.Main-6}
              \sup_{t>0}|\mathfrak{B}_k^{(\beta,t)}(f)(x)| \le c\mathcal{M}_k(f)(x),
          \end{align}
           where $\mathcal{M}_k(f)$ is the maximal function associated with Dunkl transform.\\ 
          \item  For $ f \in L_k^p(\mathbb{R}^n),\, 1\le p < \infty, $ 
          \begin{align*}
              \sup_{x\in \mathbb{R}^n}|\mathfrak{B}_k^{(\beta,t)}(f)(x)|\le \Tilde{c}t^{-\frac{2}{p\beta}(n+2\gamma-1)}\|f\|_{L_k^p(\mathbb{R}^n)}, 
          \end{align*} where $\Tilde{c}=c_k \left( \mathcal{W}_k^{(\beta,1)}(0)\right)^{1/p} \|\mathcal{W}_k^{(\beta,1)}\|_{L_k^1(\mathbb{R}^n)}^{1/q}.$
          \item For a fixed $\beta >0$, $\mathfrak{B}_k^{(\beta,t)}$ satisfies
          \begin{align} \label{beta-additivity}
              \mathfrak{B}_k^{(\beta,t)}(\mathfrak{B}_k^{(\beta,\tau)})=\mathfrak{B}_k^{(\beta,t+\tau)},\quad \text{  for all $t, \tau >0.$}
          \end{align}
          \item Suppose $f \in L_k^p(\mathbb{R}^n)$ when $1\le p < \infty$ and  $f \in C_0(\mathbb{R}^n)$ when $ p=\infty$. In either case 
          \begin{align*}
              \lim_{t\rightarrow 0} \mathfrak{B}_k^{(\beta,t)}(f) (x)= f(x) . 
          \end{align*}
         \item For $\beta, t >0$, we have
         \begin{align} \label{Prop-Main-xi}
\mathfrak{B}_k^{(\beta,t)}\left(I_k^\alpha\right)=I_k^\alpha\left(\mathfrak{B}_k^{(\beta,t)}\right).
         \end{align}
      \end{enumerate}
\end{proposition}
\begin{proof}
   The property $(i)$ is an easy consequence of the definition of $\mathcal{W}_k^{(\beta,t)}$ and the substitution  $y^\prime=s^{1/\beta}y$. 
  The positivity of $\mathcal{W}_k^{(\beta,t)}$ is straightforward for $\beta=1$, $2$. For the case $0<\beta<2$, the result is proved in \cite{Rejeb}. When $\beta=2m$, $m \in \mathbb{N}$, the function $e^{-t\|\xi\|^\beta}$ belongs to the class of Schwartz functions, and the invariance of $\mathcal{S}(\mathbb{R}^n)$ under the Dunkl transform gives the decay property to $\mathcal{W}_k^{(\beta, t)}$. Since $e^{-t\|\cdot\|^\beta} \in L_k^1(\mathbb{R}^n)$, the Riemann–Lebesgue lemma implies that $\mathcal{W}_k^{(\beta,t)} \in \mathcal{C}_0(\mathbb{R}^n)$. The property $(iii)$ partially follows from these observations. Now, for the non-even values of $\beta$, we follow the method of Aliev and Rubin \cite{Aliev-2008}. Invoking the relation between the Dunkl and Hankel transform of radial function \cite{Kobayashi, Rosler-2008}, we have 
    \begin{align}
        \mathcal{W}_k^{(\beta,t)}(\xi) &= 2^{-(\gamma+\frac{n}{2}-1)} \mathcal{H}_{2,\gamma+\frac{n}{2}-1} (e^{-tr^{\beta}})(\|\xi\|) \notag \\
        &=  {r}^{-v} \int_0^\infty e^{-ts^{\beta}} J_v(rs)s^{v+1}ds, \label{Dunkl-Hankel}
    \end{align} where $r=\|\xi\|$ and $v=\gamma+\frac{n}{2}-1$. For convenience, we denote the above integral by $W(r)$. Using the derivative of the Bessel function, we obtain the following identity: 
    \begin{align*}
        W(r)= r^{-(2v+1)}\int_0^\infty e^{-ts^{\beta}} \frac{d}{ds}\left( (rs)^{v+1} J_{v+1}(rs)\right)ds.
    \end{align*} The integration by parts and the substitution $z=(rs)^\beta$ will give
    \begin{align*}
        W(r)= t r^{-(2v+\beta+1)}\int_0^\infty e^{-tr^{-\beta} z}J_{v+1}(z^{1/\beta}) z^{\frac{v+1}{\beta}}dz.
        \end{align*} 
        Now $ r^{2v+\beta+1} W(r) = t I_{r^{-\beta}}$, where $I_\delta=\int_0^\infty e^{-t\delta z}J_{v+1}(z^{1/\beta}) z^{\frac{v+1}{\beta}}dz$. The calculation will be finalized by the value of $I_0=\lim_{\delta\rightarrow 0}I_\delta.$ A similar calculation in \cite{Aliev-2008} gives the value of $I_0= 2^{\gamma+\frac{n}{2}+\beta-2} \Gamma \left( \frac{2\gamma+n+\beta-1}{2} \right) \Gamma\left( \frac{\beta+1}{2} \right)$. Property $(iv)$ is the application of the inversion formula of the Dunkl transform. The property $(v)$ is due to the boundedness of the convolution operator. For $0<\beta \leq 2$, the constant $C(\beta)$ can be deduced from property $(iv)$, and for $\beta>2,$ the unified constant for all $t>0$ is obtained by a change of variable. The property $(vi)$ is a direct result of \cite[Theorem 6.2]{Thangavelu}. For property $(vii),$ we consider the relation
        $ \mathcal{W}_k^{(\beta,t)}(\xi)=t^{-\frac{1}{\beta}(n+2\gamma)}W_k^\beta(t^{-1/\beta}\xi),$ where $W_k^\beta(\xi)=\mathcal{W}_k^{(\beta,1)}(\xi).$ The required bound in $t$ can be achieved by applying the convolution definition of $\beta$-semigroup, H\"olders inequality, and Theorem \ref{C.Trns}. Property $(viii)$ can be deduced from a straightforward application of the Dunkl transform over the convolution product.  Property $(ix)$ can be deduced from \cite[Theorem 4.2.]{Thangavelu}.     In the end, the commutativity relation $(x)$ follows from the definition.
\end{proof}
\begin{remark}
     From properties $(viii) $ and $(ix)$ of the above proposition, we see that the convolution operator $\mathfrak{B}_k^{(\beta,t)}$  satisfies the semigroup property. So, we call it a $\beta$-semigroup.
 \end{remark}
\subsection{Inversion of Dunkl-Riesz potential} We derive an inversion formula for the Dunkl-Riesz potential associated with the general reflection group. In \cite{Liu}, Lie et al. established the inversion formula for the special case $W = \mathbb{Z}_2^n$. Using the methods developed in \cite{Liu, Sezer} and employing the wavelet approach \cite{Rubin-1996}, we extend these results to a more general setting.

In the following definition, we define a wavelet-like transform that plays a central role in the derivation of the inversion formula.

\begin{definition}
    A finite Borel measure $\nu$ on $[0, \infty)$, satisfying  $\nu([0, \infty)) = 0$ and $\int_0^\infty sd|\nu|(s)<\infty$ is called the wavelet measure. For $\beta, t >0$ and $f \in \mathcal{S}(\mathbb{R}^n)$, the wavelet-like transform $\mathbf{W}_k^{(\beta,t)}$ associated with $\beta$-semigroup  is defined as 
\begin{align} \label{wavlet}
\mathbf{W}_k^{(\beta,t)} (f)(\xi) = \int_0^\infty \mathfrak{B}_k^{(\beta, st)}(f)(\xi) d\nu(s).
\end{align}
\end{definition} 
The impending result discusses the continuity of the wavelet-like transform over $L_k^p(\mathbb{R}^n)$.
\begin{proposition}
    Let $\beta , t >0$ and $f\in L_k^p(\mathbb{R}^n)$, for $1\leq p \leq \infty$. Then 
    $\mathbf{W}_k^{(\beta, t)}$ is a bounded operator on $L_k^p(\mathbb{R}^n)$  with $$\|\mathbf{W}_k^{(\beta,t)} (f)\|_{L_k^p(\mathbb{R}^n)} \leq C({\beta}) \|\nu\| \|f\|_{L_k^p(\mathbb{R}^n)},$$ where $C(\beta)$ is given in \eqref{prop-main-5}.
\end{proposition}
\begin{proof}
    The proof is a consequence of Minkowski's integral inequality and \eqref{prop-main-5-0}. 
\end{proof}
\noindent We provide a few lemmas and definitions to establish the inversion formula for the Riesz potential using $\mathbf{W}_k^{(\beta,t)}$. In the sequel, we consider $\nu$ as a wavelet measure.
\begin{lemma} \label{Lemma-R-1}
     Let $0< \alpha < n+2\gamma$, $\beta >0 $, $f\in L_k^p(\mathbb{R}^n)$ with $1\leq p < \frac{n+2\gamma}{\alpha}$ and $g=\mathcal{I}_k^\alpha(f)$. Then $\mathbf{W}_k^{(\beta, t)}$ can be represented as 
     \begin{align*}
         \mathbf{W}_k^{(\beta, t)}(g)(\xi)= \frac{1}{\Gamma(\alpha / \beta)} \int_0^\infty \int_{0}^\infty (r-st)_+^{\frac{\alpha}{\beta}-1} \mathfrak{B}_k^{(\beta, r)}(f)(\xi) dr d\nu(s),
    \end{align*}  
 where $ (r-st)_+ =
\begin{cases} 
  0 & \text{ if } r-st \leq 0,\\ 
  r-st & \text{ if } r-st \geq 0.   
\end{cases} $ 
\end{lemma}
\begin{proof}
    Considering \eqref{Prop-Main-xi} in \eqref{wavlet}, we obtain
    \begin{align*}
        \mathbf{W}_k^{(\beta, t)}(g)(\xi)&= \int_0^\infty \mathfrak{B}_k^{(\beta, st)}\left(\mathcal{I}_k^\alpha(f)\right)(\xi) d\nu(s) =\int_0^\infty \mathcal{I}_k^\alpha \left(\mathfrak{B}_k^{(\beta, st)}(f)\right)(\xi) d\nu(s).
    \end{align*} 
    Using the identities \eqref{Riesz-beta}, \eqref{beta-additivity} and the substitution $r=\tau+st$, we have the desired result
    \begin{align*}
        \int_0^\infty \mathcal{I}_k^\alpha\left( \mathfrak{B}_k^{(\beta, st)}(f)\right)(\xi) d\nu(s) &= \frac{1}{\Gamma(\alpha / \beta)}
        \int_0^\infty \int_{st}^\infty (r-st)^{\frac{\alpha}{\beta}-1} \mathfrak{B}_k^{(\beta, r)}(f)(\xi) dr d\nu(s)\\
        & = \frac{1}{\Gamma(\alpha / \beta)} \int_0^\infty \int_{0}^\infty (r-st)_+^{\frac{\alpha}{\beta}-1} \mathfrak{B}_k^{(\beta, r)}(f)(\xi) dr d\nu(s).
    \end{align*}  
\end{proof}
\begin{definition}
  For $\epsilon >0$ and $f \in \mathcal{S}(\mathbb{R}^n)$, we define the truncated operator $\mathfrak{T}_\epsilon$ as 
  \begin{align}\label{truncated function-1}
      \mathfrak{T}_\epsilon (f)(\xi)=  \int_\epsilon^\infty
       \mathbf{W}_k^{(\beta,t)}(f)(\xi) t^{-\frac{\alpha}{\beta}-1}dt.
  \end{align}
\end{definition}
\begin{lemma}\label{Lemma-R-2}
    Let $0< \alpha < n+2\gamma$, $\beta >\alpha $, $f\in L_k^p(\mathbb{R}^n)$ with $1\leq p < \frac{n+2\gamma}{\alpha}$ and $g=\mathcal{I}_k^\alpha(f)$. Then
    \begin{align*}
        \mathfrak{T}_\epsilon (g)(\xi)= \int_0^\infty \mathcal{B}_k^{(\beta, \epsilon s)}(f)(\xi)\mathcal{K}_{\alpha/\beta}(s)ds,
    \end{align*} where $\mathcal{K}_{\theta}(s)= \frac{1}{s} \frac{1}{\Gamma(\theta+1)}\int_0^s (s-r)^\theta d\nu(r)$, with $\theta=\alpha/\beta$. 
\end{lemma}
\begin{proof}
    In view of Lemma \ref{Lemma-R-1}, the operator $ \mathfrak{T}_\epsilon$ can be written as
    \begin{align*}
         \mathfrak{T}_\epsilon\mathcal{I}_k^\alpha (f)(\xi) &= \frac{1}{\Gamma(\alpha / \beta)} \int_\epsilon^\infty 
          \int_0^\infty \int_{0}^\infty (r-st)_+^{\frac{\alpha}{\beta}-1} \mathfrak{B}_k^{(\beta, r)}(f)(\xi) dr d\nu(s) t^{-\frac{\alpha}{\beta}-1}dt.
    \end{align*}
    Then Fubini's theorem and the application of the formula 
$\int_{1}^\tau (r-t)^{\frac{\alpha}{\beta}-1} t^{-\frac{\alpha}{\beta}-1}dt =\frac{\Gamma(\alpha/\beta)}{\Gamma(\alpha/\beta +1)}\frac{(s-1)^{\alpha/\beta}}{s}$ for $s>1$ gives the required result.
\end{proof}
\begin{remark}
  The function $(I_{0^+}^{\theta+1}\nu)s=\frac{1}{\Gamma(\theta+1)}\int_0^s (s-r)^\theta d\nu(r)$ is known as the Riemann-Liouville integral of the measure $\nu$ of order $\theta+1$. For more details, one can refer to  \cite{Oldham}.
\end{remark}
Now, we establish the inversion formula for the Dunkl-Riesz potential. 
\begin{theorem} \label{Inversion-Riesz}
     Let $0< \alpha < n+2\gamma$, $\beta >0 $, $f\in L_k^p(\mathbb{R}^n)$ with $1\leq p < \frac{n+2\gamma}{\alpha}$ 
      and $\nu$ be a wavelet measure satisfying the conditions 
      \begin{align*}
          (i)\quad & \int_0^\infty s^{l}d|\nu|(s) < \infty \quad \text{for some }  l > \frac{\alpha}{\beta}, \\
          (ii)\quad  & \int_0^\infty s^m d\nu(s) = 0, \text{ 
 for} m = 0, 1, \ldots, \left\lfloor \frac{\alpha}{\beta} \right\rfloor,
\end{align*}
where \( \lfloor \cdot \rfloor \) denotes the greatest integer function.
     If $g=\mathcal{I}_k^\alpha (f),$ then 
     \begin{align*}
     \int_0^\infty  \mathbf{W}_k^{(\beta,t)}(g)(\xi) t^{-\frac{\alpha}{\beta}-1}dt = \lim_{\epsilon\longrightarrow  0}
       \mathfrak{T}_\epsilon (g)(\xi) =  C\left(\frac{\alpha}{\beta}, \nu\right)f(\xi),
     \end{align*} where 

\begin{align*}
       \mathbf{W}_k^{(\beta,t)} (f)(\xi) = \int_0^\infty  \mathfrak{B}_k^{(\beta, st)}(f)(\xi) d\nu(s),\quad 
   \mathfrak{T}_\epsilon (g)(\xi) = \int_\epsilon^\infty  \mathbf{W}_k^{(\beta,t)}(g)(\xi) t^{-\frac{\alpha}{\beta}-1}dt
\end{align*} and 
\[ 
C(r,\nu) = \int_0^{\infty} \frac{\mu(t)}{t^{1+r}}dt =  
\begin{cases} 
  \Gamma(-r)\int_0^{\infty}s^{r}d\nu(s) & \text{ if } r \notin \mathbb{N}_0 \\ 
  \frac{(-1)^{r+1}}{r !} \int_0^{\infty}s^{r} \ln{s}d\nu(s) & \text{ if } r \in \mathbb{N}_0. \\  
\end{cases} 
\] 
The function \( \mu(t) = \int_0^{\infty} e^{-ts} \, d\nu(s) \) represents the Laplace transform of the wavelet measure \( \nu \). The convergence of the associated integral expression is considered in the  \( L_k^p(\mathbb{R}^n) \)-norm. Moreover, if \( f \in \mathcal{C}_0(\mathbb{R}^n) \cap L_k^p(\mathbb{R}^n) \), the convergence is uniform on \( \mathbb{R}^n \).
\end{theorem}
\begin{proof}
In the light of Lemma \ref{Lemma-R-2}, we have 
\begin{align*}
    \mathfrak{T}_\epsilon\left(\mathcal{I}_k^\alpha (f)\right)(\xi)-C\left(\frac{\alpha}{\beta},\nu\right)f(\xi) = \int_0^\infty \left(
    \mathcal{B}_k^{(\beta, \epsilon s)}(f)(\xi) - f(\xi)
    \right)\mathcal{K}_{\alpha/\beta}(s) ds.
\end{align*}Then Minkowski's integral inequality implies 
\begin{align*}
    \|  \mathfrak{T}_\epsilon(\mathcal{I}_k^\alpha(f))-C\left(\frac{\alpha}{\beta},\nu\right) f\|_{L_k^p(\mathbb{R}^n)} \leq \int_0^\infty 
    \|
    \mathcal{B}_k^{(\beta, \epsilon s)}(f) - f
    \|_{L_k^p(\mathbb{R}^n)} |\mathcal{K}_{\alpha/\beta}(s)| ds.
\end{align*}
Using the Dominated convergence theorem and Proposition \ref{Proposition-main} $(ix)$, we obtain
\begin{align*}
    \|  \mathfrak{T}_\epsilon(\mathcal{I}_k^\alpha (f))-C\left(\frac{\alpha}{\beta},\nu\right)f\|_{L_k^p(\mathbb{R}^n)} \longrightarrow 0 \quad \text{ as } \epsilon \longrightarrow 0^+. 
\end{align*} For $f\in \mathcal{C}_0(\mathbb{R}^n)\cap L_k^p(\mathbb{R}^n)$, the uniform convergence is an easy consequence of Proposition \ref{Proposition-main} $(ix)$.   
\end{proof}
\begin{remark}
    The inversion of the Dunkl-Riesz potential can be derived under minimal conditions on $\nu$ by assuming that the parameter $\beta$ takes a value greater than $\alpha$. Consequently, $\left\lfloor \frac{\alpha}{\beta} \right\rfloor = 0$, ensuring that the prior conditions align with those imposed on $\nu$ itself.
\end{remark}

\subsection{Dunkl-Riesz potential space:} We define the Dunkl-Riesz potential space for $0<\alpha< n+2\gamma$, and $1<p < \frac{n+2\gamma}{\alpha}$  as 
 \begin{align*}
     \mathcal{I}_k^\alpha(L_k^p(\mathbb{R}^n))=\{g : g = \mathcal{I}_k^\alpha(f), \text{ for } f\in L_k^p(\mathbb{R}^n) \} .
 \end{align*} 
Even though the space $\mathcal{I}_k^\alpha\left(L_k^p(\mathbb{R}^n)\right) \subset L_k^q(\mathbb{R}^n),$ where $\frac{1}{q}=\frac{1}{p}-\frac{\alpha}{n+2\gamma}$ \cite{Hassani}, we impose a new structure on  $\mathcal{I}_k^\alpha(L_k^p(\mathbb{R}^n))$  by defining a norm as $\|g\|_{\mathcal{I}_k^\alpha(L_k^p(\mathbb{R}^n))}= \|f\|_{L_k^p(\mathbb{R}^n)},$ which makes  $\mathcal{I}_k^\alpha(L_k^p(\mathbb{R}^n))$ a complete space.\\

\noindent We shall first prove some preliminary lemmas for characterizing the Dunkl-Riesz potential space.

\begin{lemma} \label{Lemma-C-1}
    Let $0< \alpha <n+2\gamma$, $\beta >\alpha$ and $f\in \mathcal{S}(\mathbb{R}^n)$. Then $\mathfrak{T}_{\epsilon}(f)$ admits the convolution representation
    \begin{align*}
      \mathfrak{T}_{\epsilon}(f) (\xi)   = f \underset{k}{\ast} \mathcal{W}_{\epsilon}(\xi),
    \end{align*} where $\mathcal{W}_{\epsilon}(\xi)= \int_{\epsilon}^\infty \int_0^\infty \mathcal{W}_k^{(\beta, st)}(\xi) d\nu(s)t^{-(\frac{\alpha}{\beta}+1)}dt$.
\end{lemma}
\begin{proof} We begin the proof by considering the Dunkl transform of $\mathfrak{T}_{\epsilon}(f)$
    \begin{align*}
        \mathcal{D}_k(\mathfrak{T}_{\epsilon}(f))(\xi)= c_k \int_{\mathbb{R}^n} \mathfrak{T}_{\epsilon}(f)(y)E_k(-i\xi,y)w_k(y)dy.
    \end{align*}
 Taking into account \eqref{truncated function-1}, \eqref{wavlet}, and \eqref{beta-convolution}, we obtain
 \begin{align*}
     \mathcal{D}_k(\mathfrak{T}_{\epsilon}(f))(\xi) &= c_k  \int_{\mathbb{R}^n}  \int_{\epsilon}^\infty \int_0^\infty  
    \mathcal{W}_k^{(\beta,st)}\underset{k}{\ast}
    f(y)E_k(-i\xi,y)d\nu(s)t^{-(\frac{\alpha}{\beta}+1)}dt w_k(y)dy. 
 \end{align*}
 Fubini's theorem implies
 \begin{align} \label{Truncation-1}
     \mathcal{D}_k(\mathfrak{T}_{\epsilon}f)(\xi)=\mathcal{D}_k(f)(\xi) \int_{\epsilon}^\infty \int_0^\infty e^{-st\|\xi\|^\beta}  d\nu(s)t^{-(\frac{\alpha}{\beta}+1)}dt.
 \end{align} We denote $T_\epsilon(\xi)=  \int_{\epsilon}^\infty \int_0^\infty e^{-st\|\xi\|^\beta}  d\nu(s)t^{-(\frac{\alpha}{\beta}+1)}dt,$ and observe that 
 $T_\epsilon$ is a radial function with
 $\|T_\epsilon\|_{L_k^{\infty}(\mathbb{R}^n)} \leq \frac{\beta}{\alpha} \|\nu\|  \epsilon^{-\frac{\alpha}{\beta}}$. Now, it is enough to prove that $T_\epsilon \in L_k^1(\mathbb{R}^n)$. We consider the function 
 \begin{align*}
     \int_{\mathbb{R}^n} \int_{\epsilon}^\infty& \int_0^\infty e^{-st\|\xi\|^\beta}  d|\nu|(s)t^{-(\frac{\alpha}{\beta}+1)}dt w_k(\xi)d\xi \\ &= d_k\int_0^\infty \int_{\epsilon}^\infty \int_0^\infty e^{-st r^\beta}  d|\nu|(s)t^{-(\frac{\alpha}{\beta}+1)}dt r^{n+2\gamma-1}dr\\
     & = d_k \int_{\epsilon}^\infty \int_0^\infty \int_0^\infty 
     e^{-st r^\beta} r^{n+2\gamma-1}dr d|\nu|(s)t^{-(\frac{\alpha}{\beta}+1)}dt, 
 \end{align*} where we have used Tonelli's theorem for interchanging the integral. We can find large enough $N_1, N_2\in \mathbb{N}$ such that  $e^{-st r^\beta} r^{n+2\gamma-1} \leq r^{-N_2}$ fo all $r\geq N_1$. 
 \begin{align*}
     \int_{\epsilon}^\infty \int_0^\infty& \int_0^\infty 
     e^{-st r^\beta} r^{n+2\gamma-1}dr d|\nu|(s)t^{-(\frac{\alpha}{\beta}+1)}dt \\
     & = \int_{\epsilon}^\infty \int_0^\infty \int_0^{N_1} + \int_{\epsilon}^\infty \int_0^\infty \int_{N_1}^\infty
     e^{-st r^\beta} r^{n+2\gamma-1}dr d|\nu|(s)t^{-(\frac{\alpha}{\beta}+1)}dt\\
     & \leq   \frac{\beta \epsilon^{-\frac{\alpha}{\beta}} }{\alpha}  \|\nu\| \left(   \frac{N_1^{n+2\gamma}}{n+2\gamma} + \frac{N_1^{-N_2+1}}{1-N_2}
     \right),
 \end{align*} which gives the integrability of $T_\epsilon$. Moreover,
 \begin{align*}
      \mathcal{D}_k^{-1}(T_\epsilon)(\xi) &= c_k \int_{\mathbb{R}^n} T_\epsilon(y)E_k(i\xi,y) w_k(y)dy \\
      &= \int_{\epsilon}^\infty \int_0^\infty \left( c_k \int_{\mathbb{R}^n} e^{-st\|y\|^\beta}  E_k(i\xi,y) w_k(y)dy\right) d\nu(s)t^{-(\frac{\alpha}{\beta}+1)}dt\\
      &= \int_{\epsilon}^\infty \int_0^\infty \mathcal{W}_k^{(\beta, st)}(\xi) d\nu(s)t^{-(\frac{\alpha}{\beta}+1)}dt = \mathcal{W}_{\epsilon}(\xi).
 \end{align*} Since $\mathcal{D}_k$ is an isomorphism on $L_k^2(\mathbb{R}^n)$ we conclude that $\mathcal{W}_{\epsilon} \in L_k^2(\mathbb{R}^n)$ and \eqref{Truncation-1} becomes  $ \mathcal{D}_k(\mathfrak{T}_{\epsilon}(f))(\xi)=\mathcal{D}_k(f)(\xi) \mathcal{D}_k(\mathcal{W}_{\epsilon})(\xi)$. Thus, we have $\mathfrak{T}_{\epsilon}(f)= f \underset{k}{\ast}\mathcal{W}_\epsilon$.  
\end{proof}
\begin{lemma} \label{Lemma-c-2}
    Let $f \in \mathcal{S}(\mathbb{R}^n)$ and $\epsilon >0$. Then $\langle \mathfrak{T}_{\epsilon}(f), h \rangle = \langle f, \mathfrak{T}_{\epsilon}(h) \rangle$ for every $h\in \mathcal{S}(\mathbb{R}^n)$, where $ \langle \mathfrak{T}_{\epsilon}(f), h \rangle= \int_{\mathbb{R}^n} \mathfrak{T}_{\epsilon}(f)(\xi)h(\xi)w_k(\xi)d\xi$, in the weak sense.
\end{lemma}
\begin{proof}
 Considering $\langle \mathfrak{T}_{\epsilon}(f), h \rangle$ and, using Lemma \ref{Lemma-C-1} and Proposition \ref{D.C-Property}, we achieve
  \begin{align*}
       \int_{\mathbb{R}^n} \mathfrak{T}_{\epsilon}(f)(\xi)h(\xi)w_k(\xi)d\xi &= \int_{\mathbb{R}^n}  f \underset{k}{\ast}\mathcal{W}_\epsilon(\xi)h(\xi)w_k(\xi)d\xi\\
       &= \int_{\mathbb{R}^n} h(\xi) \int_{\mathbb{R}^n} \mathcal{D}_k(f)(x)\mathcal{D}_k(\mathcal{W}_{\epsilon})(x) E_k(i\xi,x)  w_k(x)dx w_k(\xi)d\xi\\
       &= \int_{\mathbb{R}^n}  \mathcal{D}_k(f)(x)\mathcal{D}_k(\mathcal{W}_{\epsilon})(x) \mathcal{D}_k(h)(-x)w_k(x)dx.
  \end{align*} Similar arguments will imply 
  \begin{align*}
      \int_{\mathbb{R}^n} f(\xi) \mathfrak{T}_{\epsilon}(h)(\xi)w_k(\xi)d\xi &=  \int_{\mathbb{R}^n}  \mathcal{D}_k(f)(-x)\mathcal{D}_k(\mathcal{W}_{\epsilon})(x) \mathcal{D}_k(h)(x)w_k(x)dx.
  \end{align*}Thus, the required identity is obtained by substituting $y=-x$. 
\end{proof}
\noindent We recall the Lizorkin space in the Dunkl setting, which is crucial in characterizing Dunkl-Riesz potential spaces. For more details regarding Lizorkin space in the Dunkl setting, we refer \cite{Gorbachev-2021}.
\begin{definition}
    The Lizorkin space associated with the Dunkl transform is defined by 
    \begin{align*}
        \Phi_k = \{ f\in \mathcal{S}(\mathbb{R}^n); \int_{\mathbb{R}^n} \xi^m f(\xi)w_k(\xi)d\xi =0, \quad \forall m \in \mathbb{N}^n 
        \}. 
    \end{align*}
\end{definition} 
In \cite{Gorbachev-2021}, it is proved that 
$\mathcal{I}_k^\alpha (\Phi_k)=\Phi_k$, consequently one can define the Dunkl-Riesz potential for $\Phi_k^\prime$ with the relation  \begin{align*}
    \langle \mathcal{I}_k^\alpha(f), \phi \rangle = \langle f, \mathcal{I}_k^\alpha(\phi) \rangle , \quad \text{ for } \phi \in \Phi_k. 
\end{align*}
The forthcoming theorem characterizes the Dunkl-Riesz potential space.
\begin{theorem}\label{Char.R.P.Space}
    Let $0<\alpha<n+2\gamma, \quad 1<p<\frac{n+2\gamma}{\alpha}$, $\beta >\alpha $ and $\nu$ be the wavelet measure with $C\left(\frac{\alpha}{\beta},\nu\right)\neq 0.$
    Then $g \in \mathcal{I}_k^\alpha(L_k^p(\mathbb{R}^n))$ if and only if $g \in L_k^q(\mathbb{R}^n)$ for  $q= \frac{p(n+2\gamma)}{n+2\gamma-p\alpha}$ and 
    $\underset{\epsilon>0}{\sup} \| \mathfrak{T}_\epsilon (g)\|_{L_k^p(\mathbb{R}^n)} < \infty.$
\end{theorem}
\begin{proof}
    The forward implication follows from the boundedness of the Riesz potential \cite{Hassani}.  For the converse part, we use the techniques in \cite{Sezer}.
    \\ Consider $\underset{\epsilon>0}{\sup} \| \mathfrak{T}_\epsilon (g)\|_{L_k^p(\mathbb{R}^n)} < \infty$. According to the Banach-Alaoglu theorem we have $f\in L_k^p(\mathbb{R}^n)$ such that $\lim_{\epsilon_m \rightarrow 0}\langle  \mathfrak{T}_{\epsilon_m} (g) , \phi \rangle = \langle f, \phi \rangle$ for $\phi \in \Phi$. In view of Lemmas \ref{Lemma-R-2}, \ref{Lemma-C-1}, and \ref{Lemma-c-2} and Theorem \ref{Inversion-Riesz}, we have 
    \begin{align*}
        \langle \mathcal{I}_k^\alpha(f), \phi \rangle & = \langle f , \mathcal{I}_k^\alpha(\phi) \rangle = {C\left(\frac{\alpha}{\beta},\nu\right)}^{-1}\underset{\epsilon_m \rightarrow 0}{\lim} \langle \mathfrak{T}_{\epsilon_m}(g),  \mathcal{I}_k^\alpha(\phi) \rangle =  {C\left(\frac{\alpha}{\beta},\nu\right)}^{-1}\underset{\epsilon_m \rightarrow 0}{\lim} \langle g, \mathfrak{T}_{\epsilon_m}(
        \mathcal{I}_k^\alpha(\phi)) \rangle \\
        & = \underset{\epsilon_m \rightarrow 0}{\lim} 
       \Biggl \langle   g,  {C\left(\frac{\alpha}{\beta},\nu\right)}^{-1} \int_0^\infty \mathcal{B}_k^{(\beta, s\epsilon_m )}(\phi) \mathcal{K}_{\alpha/\beta}(s) ds
      \Biggl  \rangle .
    \end{align*}
We consider $   \left|  \Bigl \langle   g,  {C\left(\frac{\alpha}{\beta},\nu\right)}^{-1} \int_0^\infty \mathcal{B}_k^{(\beta, s\epsilon_m )}(\phi) \mathcal{K}_{\alpha/\beta}(s) ds
      \Bigl  \rangle -\langle g, \phi \rangle
      \right| $, and using H\"olders inequality and Minkowski's inequality, we deduce 
 \begin{align*}
   &\left|  \Biggl \langle   g,  {C\left(\frac{\alpha}{\beta},\nu\right)}^{-1} \int_0^\infty \mathcal{B}_k^{(\beta, s\epsilon_m )}(\phi) \mathcal{K}_{\alpha/\beta}(s) ds
      \Biggl  \rangle -\langle g, \phi \rangle
      \right|\\
      &\leq {C\left(\frac{\alpha}{\beta},\nu\right)}^{-1} \|g\|_{L_k^q(\mathbb{R}^n)} \int_0^\infty \| \mathcal{B}_k^{(\beta, s\epsilon_m)}(\phi)-\phi\|_{L_k^p(\mathbb{R}^n)} \left| \mathcal{K}_{\alpha/\beta}(s) 
      \right|ds,
 \end{align*}where $\frac{1}{p}+\frac{1}{q}=1$.  Dominated convergence theorem implies that $\langle \mathcal{I}_k^\alpha( f), \phi \rangle=\langle g, \phi \rangle$
for $\phi \in \Phi$. Thus, $\mathcal{I}_k^\alpha (f) - g \perp \Phi$, consequently $\mathcal{I}_k^\alpha(f)= g+ P$, for some polynomial $P(x)$. Since $\mathcal{I}_k^\alpha(f), g \in L_k^q(\mathbb{R}^n)$, we have $P(x)\in L_k^q(\mathbb{R}^n)$, thus $P=0$ and $g = \mathcal{I}_k^\alpha(f)$. 
\end{proof}
\section{Potentials associated with Dunkl-Laplacian } \label{S:4}
In this section, we introduce the bi-parametric potential in the Dunkl setting, which extends the Dunkl-Bessel and 
Dunkl-Flett potentials. This generalization enables a broader exploration of potential theoretic concepts in spaces with reflection symmetries. The key contributions of this section include deriving an inversion formula for the bi-parametric potential and establishing a characterization of the corresponding potential spaces.\\
We have observed that the $\beta$-semigroup provides a unified framework for the representation of the $k$-heat and $k$-Poisson semigroups. Using this observation, we can express the Dunkl-Bessel and Dunkl-Flett potentials in terms of the $\beta$-semigroup, as demonstrated below:
\begin{align*}
    \mathcal{J}_k^\alpha(f)(\xi) = \frac{1}{\Gamma(\frac{\alpha}{2})}\int_0^\infty t^{\frac{\alpha}{2}-1}e^{-t}\mathfrak{B}_k^{(2,t)}(f)(\xi)dt,\\
    \mathcal{F}_k^\alpha (f)(\xi) = \frac{1}{\Gamma(\alpha)}\int_0^\infty t^{{\alpha}-1}e^{-t}\mathfrak{B}_k^{(1,t)}(f)(\xi)dt.
\end{align*}
The interplay between the powers of $t$ and the values of $\beta$ within the semigroup $\mathfrak{B}_k^{(\beta,t)}$, motivates us to introduce a more general class of potential operators in the Dunkl setting, termed as bi-parametric potential, analogous to classical setting \cite{Aliev}. We defined the bi-parametric potential as
\begin{align} \label{bi-potential}
\mathfrak{S}_k^{(\alpha, \beta)} (f)(\xi) = \frac{1}{\Gamma(\alpha/\beta)} \int_0^\infty t^{\frac{\alpha}{\beta}-1}e^{-t} \mathfrak{B}_k^{(\beta,t)}(f)(\xi)dt,
\end{align} where $\alpha, \beta >0$ and $f\in \mathcal{S}(\mathbb{R}^n)$. This formulation effectively unifies the definition of potential operators by combining the flexibility of the $\beta$-semigroup with a weight factor of the form $t^{\frac{\alpha}{\beta}-1}e^{-t}$. The boundedness of the bi-parametric potential operator can be established by using the boundedness of the Dunkl convolution product and Minkowski's integral inequality. Moreover we have 
\begin{align*} 
    \|\mathfrak{S}_k^{(\alpha, \beta)} (f)\|_{L_k^p(\mathbb{R}^n)} \leq \| \mathcal{W}_k^{(\beta,1)}\|_{L_k^1(\mathbb{R}^n)}\|f\|_{L_k^P(\mathbb{R}^n)}\quad \text{  for $1\leq p \leq \infty$.}
\end{align*}
  The boundedness property provides a foundation for further exploration of its properties. Next, we examine the action of the Dunkl transform on the bi-parametric potential $\mathfrak{S}_k^{(\alpha, \beta)}$ as
\begin{align*}
   \mathcal{D}_k\left(\mathfrak{S}_k^{(\alpha,  \beta)} (f)\right)(\xi)& = c_k \int_{\mathbb{R}^n} \mathfrak{S}_k^{(\alpha,  \beta)} (f)(x) E_k(-i\xi,x) w_k(x)dx \\
    &= \frac{c_k}{\Gamma(\alpha/\beta)} \int_0^\infty \int_{\mathbb{R}^n} \mathfrak{B}_k^{(\beta,t)}(f)(x) E_k(-i\xi,x) w_k(x)dx\, t^{\frac{\alpha}{\beta}-1}e^{-t}dt\\
    &= \frac{\mathcal{D}_k(f)(\xi)}{\Gamma(\alpha/\beta)} \int_0^\infty e^{-t(1+\|\xi\|^\beta)} t^{\frac{\alpha}{\beta}-1} dt.
\end{align*} Using the substitution $u=t(1+\|\xi\|^\beta)$, we obtain 
\begin{align*}
    \mathcal{D}_k\left(\mathfrak{S}_k^{(\alpha, \beta)} (f)\right)(\xi)& =  \frac{\mathcal{D}_k(f)(\xi)}{(1+\|\xi\|^\beta)^{\alpha/\beta}}.
\end{align*}
On the other hand, for every $\alpha, \beta>0$ and $f \in \mathcal{S}(\mathbb{R}^n)$,  the following holds:
\begin{align*}
    \mathcal{D}_k \left( \left(I+\left({-\Delta_k}\right)^{\frac{\beta}{2}} \right)^{-\alpha/\beta}(f)\right)(\xi)= \frac{\mathcal{D}_k(f)(\xi)}{(1+\|\xi\|^\beta)^{\alpha/\beta}}.
\end{align*}
This observation leads to the conclusion that the operator$\left(I+\left({-\Delta_k}\right)^{\frac{\beta}{2}}\right)^{-\alpha/\beta}$ can be expressed in the integral form provided in \eqref{bi-potential}. 
In addition to the integral representation and the closed-form expression in terms of powers of the Dunkl Laplacian, the bi-parametric potential can also be characterized as a convolution product involving the radial kernel $\mathcal{W}_k^{(\beta,t)}$ as 
\begin{align*}
    \mathfrak{S}_k^{(\alpha, \beta)}(f)(\xi) = f\underset{k}{\ast} \mathcal{S}_k^{(\alpha, \beta)}(\xi), \text{  where  \quad }  \mathcal{S}_k^{(\alpha, \beta)}(y) = \frac{1}{\Gamma(\alpha/\beta)}\int_0^\infty t^{\frac{\alpha}{\beta}-1}e^{-t}\mathcal{W}_k^{(\beta,t)}(y)dt.
\end{align*}
To provide more clarity, we compute the Dunkl transform of $\mathcal{S}_k^{(\alpha, \beta)}$
\begin{align*}
    \mathcal{D}_k\left( \mathcal{S}_k^{(\alpha, \beta)} \right)(\xi) &= c_k \int_{\mathbb{R}^n} \mathcal{S}_k^{(\alpha, \beta)}(y) E_k(-i\xi,y)w_k(y)dy \\
    & = \int_0^\infty t^{\frac{\alpha}{\beta}-1}e^{-t}\, c_k \int_{\mathbb{R}^n} \mathcal{W}_k^{(\beta,t)}(y) E_k(-i\xi,y)w_k(y)dy\, dt
\end{align*} 
Using the property $\mathcal{D}_k\left( \mathcal{W}_k^{(\beta,t)} \right)(x)=e^{-t\|x\|^\beta}$ and the substitution $u=t(1+\|\xi\|^\beta),$ we deduce $\mathcal{D}_k\left( \mathcal{S}_k^{(\alpha, \beta)} \right)(\xi) = (1+\|\xi\|^\beta)^{-\alpha/\beta}$.  \\
Now, we turn our attention to the inversion of the bi-parametric potential using the wavelet method, following the classical setting in \cite{Sezer}. We refer to \cite{Ben-Said-2022, Ben, Kallel} for the inversion of the Dunkl Bessel and Flett potentials using the wavelet-like transform associated with the $k$-heat and $k$-Poisson semigroups, respectively. Here, we extend the same techniques by replacing the $k$-heat and $k$-Poisson semigroup with $\mathfrak{B}_k^{(\beta,t)}$, which is suitable for the bi-parametric potential. First, we define the wavelet-like transform in this setting, referencing its formulation in \eqref{wavlet} and following the same approach to derive the inversion. One can also find the classical framework, as detailed in \cite{Aliev}. However, appropriate modifications have been made to incorporate similar arguments to the Dunkl setting.

\begin{definition} Let $f\in L_k^p(\mathbb{R}^n)$. The wavelet-like transform associated with $\mathfrak{B}_k^{(\beta,t)}$ is given by 
\begin{align} \label{Wavelet-2}
    \mathcal{V}_{k}^{(\beta,t)}(f)(\xi)= \int_0^\infty e^{-st} \mathfrak{B}_k^{(\beta,st)}(f)(\xi) d\nu(s), \text{   for  } x\in \mathbb{R}^n \text{  and  } t>0.
\end{align}
\end{definition}
For every $t>0$, the wavelet-like transform $\mathcal{V}_k^{(\beta,t)}$ is a bounded operator on $L_k^p(\mathbb{R}^n)$ to itself for $1\leq p \leq\infty$. Moreover, it satisfies the norm inequality:
$$\| \mathcal{V}_k^{(\beta,t)}(f)\|_{L_k^p(\mathbb{R}^n)} \leq \|\mathcal{W}_k^{(\beta,1)}\|_ {L_k^1(\mathbb{R}^n)}\|f\|_{L_k^p(\mathbb{R}^n)}\|\nu\|.$$
\begin{lemma} \label{Lemma-P-1}
     Let $\alpha, \beta>0$,  $f\in L_k^p(\mathbb{R}^n)$ with $1\leq p \leq \infty$ and $g=\mathfrak{S}_k^{(\alpha, \beta)}(f)$. Then for a wavelet measure $\nu$, $\mathcal{V}_k^{(\beta, t)}$ admits the representation
     \begin{align*}
         \mathcal{V}_k^{(\beta, t)}(g)(\xi)= \frac{1}{\Gamma(\alpha / \beta)} \int_0^\infty \int_{0}^\infty (r-st)_+^{\frac{\alpha}{\beta}-1} e^{-r}\mathfrak{B}_k^{(\beta, r)}(f)(\xi) dr d\nu(s),
    \end{align*}  
 where $ (r-st)_+$   is given in Lemma \ref{Lemma-R-1}.
\end{lemma}

The upcoming proposition provides an alternative representation of the bi-parametric potential, formulated in terms of the wavelet-like transform.
\begin{proposition} \label{b-potential and V}
    Let $\alpha, \beta >0$, $f \in L_k^p(\mathbb{R}^n)$ for $1\leq p \leq \infty$ and $\nu$ be the finite Borel measure measure on $[0,\infty)$ such that
    \begin{align*}
        \int_0^\infty s^{-\frac{\alpha}{\beta}}d|\nu|(s) < \infty \quad \text{and }\quad  \kappa_\nu(\alpha/\beta) = \int_0^\infty s^{-\frac{\alpha}{\beta}}d\nu(s) \neq 0.
    \end{align*} Then 
    \begin{align*}
        \mathfrak{S}_k^{(\alpha, \beta)}(f)(x) = \frac{1}{\kappa_\nu(\alpha/\beta) \Gamma(\alpha/\beta))}\int_0^\infty s^{\frac{\alpha}{\beta}-1}\mathcal{V}_k^{(\beta,s)}(f)(x) ds.
    \end{align*}
\end{proposition}
\begin{proof}
    Making use of \eqref{Wavelet-2} and Fubini's theorem, we have the required identity.
\end{proof}
\begin{definition}
    Let $\epsilon >0$ and $f\in \mathcal{S}(\mathbb{R}^n)$. We define the truncated operator $\mathcal{V}_{\epsilon}$ as 
    \begin{align*}
        \mathcal{V}_\epsilon (f)(x) = \int_\epsilon ^\infty  s^{-\frac{\alpha}{\beta}-1} \mathcal{V}_k^{(\beta,t)}\left( f\right)(x)ds.
    \end{align*}
\end{definition}
\begin{lemma}\label{Lemma-P-2}
Let $\alpha>0$, $\beta >\alpha $, $f\in L_k^p(\mathbb{R}^n)$ with $1\leq p \leq \infty$ and $g=\mathfrak{S}_k^{(\alpha, \beta)}(f)$. Then
    \begin{align*}
        \mathcal{V}_\epsilon (g)(\xi)= \int_0^\infty e^{-\epsilon s}\mathfrak{B}_k^{(\beta, \epsilon s)}(f)(\xi)\mathcal{K}_{\alpha/\beta}(s)ds,
    \end{align*} where $\mathcal{K}_{\theta}(s)$  is given in Lemma \ref{Lemma-c-2}.
\end{lemma}
We derive the inversion formula for the bi-parametric potential in the following theorem. 
\begin{theorem} \label{Thrm-Inversion of beta}
 Let \( \alpha, \beta > 0 \), \( f \in L_k^p(\mathbb{R}^n) \) for \( 1 \leq p \leq \infty \), and let \( \nu \) be a wavelet measure satisfying the conditions:
\begin{align*}
\text{(i)} & \quad \int_0^\infty s^{l} \, d|\nu|(s) < \infty \quad \text{for some } l > \frac{\alpha}{\beta}, \\
\text{(ii)} & \quad \int_0^\infty s^m \, d\nu(s) = 0 \quad \text{for all } m = 0, 1, \ldots, \left\lfloor \frac{\alpha}{\beta} \right\rfloor,
\end{align*}
where \( \lfloor \cdot \rfloor \) denotes the greatest integer function. Then the following inversion formula holds:
\[
\int_0^\infty s^{-\frac{\alpha}{\beta}-1} \, \mathcal{V}_k^{(\beta,s)}\left( \mathfrak{S}_k^{(\alpha, \beta)}(f) \right)(x) \, ds 
= \lim_{\epsilon \to 0} \mathcal{V}_\epsilon\left( \mathfrak{S}_k^{(\alpha, \beta)}(f) \right)(x) 
= C\left(\frac{\alpha}{\beta}, \nu\right) f(x),
\]
where
\[
\mathcal{V}_\epsilon\left( \mathfrak{S}_k^{(\alpha, \beta)}(f) \right)(x) := \int_\epsilon^\infty s^{-\frac{\alpha}{\beta}-1} \, \mathcal{V}_k^{(\beta,s)}\left( \mathfrak{S}_k^{(\alpha, \beta)}(f) \right)(x) \, ds,
\]
and the constant \( C\left( \frac{\alpha}{\beta}, \nu \right) \) is as defined in Theorem~\ref{Inversion-Riesz}. The convergence of the limit is to be understood in the \( L_k^p(\mathbb{R}^n) \)-norm for \( 1 \leq p < \infty \), and in the uniform norm when \( p = \infty \), provided \( f \in \mathcal{C}_0(\mathbb{R}^n) \)
\end{theorem}
\begin{proof}The proof, which resembles that of Theorem \ref{Inversion-Riesz}, follows directly from Lemma \ref{Lemma-P-2}.
\end{proof}
\subsection{Potential spaces}We define the bi-parametric potential spaces associated with the operator $\mathfrak{S}_k^{(\alpha, \beta)}$ as
\begin{align*}
    \mathfrak{S}_k^{(\alpha, \beta)}(L_k^p)=\{ g: g=\mathfrak{S}_k^{(\alpha, \beta)}(f) : f\in L_k^p(\mathbb{R}^n)\}
\end{align*}
The norm in this space is  $\|g\|_{\mathfrak{S}_{k,p}^{(\alpha, \beta)}}=\|f\|_{L_k^p(\mathbb{R}^n)}.$\\
The following theorem characterizes the bi-parametric potential spaces. 
\begin{theorem} \label{p.Inversion}
    
     Let $\alpha, \beta>0$ and $f\in L_k^p(\mathbb{R}^n)$ for $1<p<\infty$. Then $f\in \mathfrak{S}_k^{(\alpha, \beta)}(L_k^p) $ if and only if $$\underset{\epsilon>0}{\sup}\|\mathcal{V}_\epsilon (f)\|_{L_k^p(\mathbb{R}^n)} < \infty .$$
\end{theorem}
The proof will follow the same approach as in the classical work \cite{Aliev}, utilizing the lemmas provided below.
\begin{lemma} \label{Lemma-P-C-1}
    Let $ \alpha >0$, $\beta >\alpha$ and $f\in \mathcal{S}(\mathbb{R}^n)$. Then $\mathcal{V}_{\epsilon}(f)$ admits the convolution representation
    \begin{align*}
      \mathcal{V}_{\epsilon}(f)(\xi) = f \underset{k}{\ast} V_{\epsilon}(\xi),
    \end{align*} where $V_{\epsilon}(\xi)= \int_{\epsilon}^\infty \int_0^\infty e^{-st}\mathcal{W}_k^{(\beta, st)}(\xi) d\nu(s)t^{-(\frac{\alpha}{\beta}+1)}dt$.
\end{lemma}
\begin{lemma} \label{Lemma-P-c-2}
    Let $f \in \mathcal{S}(\mathbb{R}^n)$ and $\epsilon >0$. Then $\langle \mathcal{V}_{\epsilon}(f), h \rangle = \langle f, \mathcal{V}_{\epsilon}(h) \rangle$ for every $h\in \mathcal{S}(\mathbb{R}^n)$, where $ \langle \mathcal{V}_{\epsilon}(f), h \rangle= \int_{\mathbb{R}^n} \mathcal{V}_{\epsilon}(f)(\xi)h(\xi)w_k(\xi)d\xi$, in the weak sense.
\end{lemma}

\section{Rate of convergence of truncated functions} \label{S:5}
We derive the inversion formula for the Riesz potential and the bi-parametric potential by considering the limit of the truncated functions  $\mathfrak{T}_\epsilon(\mathcal{I}_k(f))$ and $\mathcal{V}_\epsilon\left( \mathfrak{S}_k^{(\alpha,\beta)}(f) \right)$, respectively, both of which are parametrized by the variable $\epsilon$. In this section, we further investigate the rate of convergence of these truncated functions having some $\eta$-smoothness at some point $x_0$. Similar studies can be found in the literature within the context of the finite difference method  (see \cite{Aliev-2014, Aliev-2013, Bayrakci-2020, Bayrakci, Eryigit}). Analogously, we investigate the rate of convergence within the framework of the wavelet method for bi-parametric potential and the Dunkl-Riesz potential, which includes the Bessel and Flett potentials as a special case.\\
We introduce the concept of smoothness for a function in terms of the modulus of continuity. The modulus of continuity 
$\eta$ is recalled in the following definition \cite{Devore}.
  \begin{definition}
      A function $\eta$ defined on the non-negative real line is called the modulus of continuity if it satisfies the properties
      \begin{enumerate}[$(i)$]
          \item $\eta(t) \rightarrow 0$ as $t \rightarrow 0$.
          \item  $\eta(t)$ is non-negative and non-decreasing on $\mathbb{R}_+$.
          \item  $\eta(t) $ is sub-additive: $\eta(t_1+t_2)\leq \eta(t_1)+\eta(t_2)$.
          \item  $\eta(t) $ is continuous on $\mathbb{R}_+$.
      \end{enumerate}
  \end{definition}
The function $\eta$ also satisfies $\eta(\lambda t)<(\lambda+1)\eta(t)$ for any $\lambda >0$. In our further study, we consider the modulus of continuity $\eta$ with the condition that $\eta(t)>at$ for some $a>0$ and $0\leq t \leq \rho$ and $\eta(t)=\eta(\rho)$ for $t> \rho$.
We define the $\eta$-smoothness property at $x_0$ in the following way.
\begin{definition}
Let $0< \rho < 1$ and $B[0,r]$ be the closed ball around the origin with radius $r$.  We say that  $f\in L_k^{2}(\mathbb{R}^n)$ has $\eta$-smoothness property at $x_0$ if 
       \begin{align} \label{smoothness-1}
            \mathcal{N}_{\eta}f(x_0) = \underset{0<r\leq \rho}{\sup} \, \frac{1}{r^{n+2\gamma}\eta(r)} \int_{B[0,r]} |\tau_{x}f(x_0)-f(x_0)|w_k(x)dx < \infty.
       \end{align} 

\end{definition}
\begin{lemma}\label{Lemmma-5-1}
    Let $f \in L_k^2(\mathbb{R}^n)$ has the $\eta$-smoothness property at $x_0$. We define the function $g$ as 
    \begin{align*}
        g(x)= |\tau_xf(x_0)-f(x_0)|, \quad \quad  x \in B[0, \rho] \text{  with  } \rho <1. 
    \end{align*} If $\phi$ is a continuously differentiable function on $[0, \rho]$, then we have
    \begin{align} \label{Inq-1}
       \left| \int_{B[0,\rho]} g(x)\phi(|x|)w_k(x)dx\right| \leq \mathcal{N}_{\eta}f(x_0)\left( \rho^{n+2\gamma}\eta(\rho)|\phi(\rho)|+\int_0^{\rho} r^{n+2\gamma} \eta(r)|\phi^\prime(r)|dr
       \right).
    \end{align}
\end{lemma}
\begin{proof} We begin the proof by
   denoting \begin{align*}
      I= \int_{B[0,\rho]} g(x)\phi(|x|)w_k(x)dx,
  \end{align*}  and substituting $x=rx^\prime$, where $r=\|x\|$ and $x^\prime= \frac{x}{\|x\|}$. As a result, we obtain 
  \begin{align*}
      I= \int_0^\rho r^{n+2\gamma-1} \phi(r)\int_{\mathbb{S}^{n-1}} g(rx^\prime)w_k(x^\prime)d\sigma(x^\prime).
  \end{align*} Introducing the functions
  \begin{align*}
      G(r)=\int_{\mathbb{S}^{n-1}} g(rx^\prime) w_k(x^\prime)d\sigma^\prime \quad \text{and} \quad Q(r)=\int_0^r G(s)s^{n+2\gamma-1}ds,
  \end{align*}we deduce that 
  \begin{align*}
      I= \int_0^\rho G(r)\phi(r)r^{n+2\gamma-1}dr = \int_0^\rho \phi(r)dQ(r)= \phi(\rho)Q(\rho)-\int_0^\rho Q(r)\phi^\prime(r)dr. 
  \end{align*} In view of \eqref{smoothness-1}, we have
\begin{align*}
    Q(r)= \int_0^r G(s)s^{n+2\gamma-1}ds = \int_{B[0,r]} g(x)w_k(x)dx  \leq r^{n+2\gamma }\eta(r)\mathcal{N}_\eta f(x_0).
\end{align*} Thus, 
\begin{align*}
   | I| \leq \mathcal{N}_{\eta}f(x_0)\left( \rho^{n+2\gamma}\eta(\rho)|\phi(\rho)|+\int_0^{\rho} r^{n+2\gamma} \eta(r)|\phi^\prime(r)|dr
       \right).
\end{align*}
\end{proof}
The forthcoming result is the particular case of the $\phi$ function in Lemma \ref{Lemmma-5-1}.
\begin{corollary} \label{corollary-5.4}
    Let $f\in L_k^2(\mathbb{R}^n)$ has $\eta$-smoothness property at $x_0$ and we denote $g(x)=|\tau_xf(x_0)-f(x_0)|$. For $0<t\leq 1, \beta \geq 1,$ and 
    $\phi_t^\beta(|x|)=\mathcal{W}_k^{(\beta,t)}(x)$, we have 
    \begin{align*}
       \left| \int_{B[0,\rho]} g(x)\phi_t^\beta(|x|) w_k(x)dx \right| \leq   
       C^1\eta(t^{\frac{1}{\beta}}) 
    \end{align*} and for  $0<\beta\leq 1$ we have 
    \begin{align*}
       \left| \int_{B[0,\rho]} g(x)\phi_t^\beta(|x|) w_k(x)dx \right| \leq C^{*}\eta(t^{\beta}), 
    \end{align*}
    where $C^1$ and $C^{*}$ are the constants independent of $t$. 
\end{corollary}
\begin{proof}
Using the relation \eqref{Dunkl-Hankel}, we have the following representation for $\phi_t^\beta$,
$$\phi_t^\beta(r)=  
    r^{-v} \int_0^\infty e^{-ts^\beta}J_v(rs)s^{v+1}ds,$$ where $r=\|x\|$ and $v=\gamma+\frac{n}{2}-1.$ In order to prove the result, we use Lemma \ref{Lemmma-5-1}. To this end, we need the bounds of $\phi_t^\beta$ and the derivative of  $\phi_t^\beta$ with respect to $r$ in terms of $t$ 
    \begin{align*}
        \phi_t^\beta(r) &=  r^{-v} \int_0^\infty e^{-ts^\beta}J_v(rs)s^{v+1}ds\\
        & = r^{-2(v+1)}\int_0^\infty e^{-t(\frac{s}{r})^\beta}J_v(s)s^{v+1}ds.
    \end{align*} Differentiating with respect to $r$, we have 
    \begin{align*}
        \frac{d}{dr}\phi_t^\beta(r) &= -
        \frac{2(v+1)}{r^{2v+3}}\int_0^\infty e^{-t(\frac{s}{r})^\beta}J_v(s)s^{v+1}ds + \frac{t\beta}{r^{2v+\beta+3}}\int_0^\infty e^{-t(\frac{s}{r})^\beta}J_v(s)s^{v+\beta+1}ds \\
        & = I_1(r)+I_2(r).
    \end{align*}
    Now, let us estimate $I_1$. We note that 
    \begin{align*}
        I_1(r) &= -\frac{2(v+1)}{r^{2v+3}}\int_0^\infty e^{-t(\frac{s}{r})^\beta}J_v(s)s^{v+1}ds 
         = -\frac{2(v+1)}{r} \phi_t^\beta(r).
    \end{align*} 
Using the fact that $\phi_t^\beta \in \mathcal{C}_0(\mathbb{R}^n)$ together with equation \eqref{Pro-Main-3}, we conclude that there exists a constant $C_1$  such that $|\phi_t^\beta(r)|\leq C_1t$. Consequently, $|I_1(r)| \leq \frac{C_2t}{r}$, where $C_2=2(v+1)C_1.$ \\
In order to estimate $I_2(r)$, we use the asymptotic behaviors of the Bessel function of the first kind. We modify $I_2$ by a change of variable and derive the bound in two different regions.
\begin{align} \label{I2-1}
    I_2(r)= \frac{t\beta}{r^{v+1}} \int_0^\infty e^{-ts^\beta} J_v(rs)s^{v+\beta+1}ds.
\end{align}
Case-1: Near infinity.\\
The function $e^{-ts^\beta} \rightarrow 0$ rapidly as $s \longrightarrow \infty$. We can find large positive real number $N_1$, $ N_2$ and a constant $C_3$ so that the following holds 
\begin{align*}
    e^{-ts^\beta}s^{v+\beta+\frac{1}{2}} \leq \frac{C_3}{s^{N_2}}, \quad \quad \text{for } s\geq N_1.
\end{align*}
Then 
\begin{align*}
    \int_{N_1}^\infty e^{-ts^\beta} J_v(rs)s^{v+\beta+1}ds & \approx  \sqrt{\frac{2}{\pi r}}\int_{N_1}^\infty e^{-ts^\beta} \cos{\left(rs-\frac{v\pi}{2}-\frac{\pi}{4}\right)} s^{v+\beta+\frac{1}{2}}ds\\
    & \leq \sqrt{\frac{2}{\pi r}}\int_{N_1}^{\infty} \frac{C_3}{s^{N_1}} ds = \frac{\sqrt{2}}{\sqrt{\pi r} (N_2+1)} \frac{C_3}{N_1^{N_2+1}} = \frac{C_4}{\sqrt{r}}. 
\end{align*}
Case-2: Over the bounded interval $(0, N_1]$. \\
Here we use the uniform bound for the Bessel function. We have
\begin{align*}
    \int_0^{N_1}e^{-ts^\beta} J_v(rs)s^{v+\beta+1}ds \leq \int_0^{N_1} s^{v+\beta+1}ds  = \frac{N_1^{v+\beta+2}}{v+\beta+2} = C_5.
\end{align*}
Combining these two cases, \eqref{I2-1} can be bounded as 
\begin{align*}
    I_2(r)&=\frac{t\beta}{r^{v+1}} \int_0^\infty e^{-ts^\beta}J_v(rs)s^{v+\beta+1} ds = \int_0^{N_1}+\int_{N_1}^\infty e^{-ts^\beta}J_v(rs)s^{v+\beta+1} ds \\
  |I_2(r)|  & \leq \frac{t\beta}{r^{v+1}} \left( \frac{C_4}{\sqrt{r}} + C_5 \right).
\end{align*}
Thus, $|(\phi_t^\beta)^{\prime}(r)| \leq  t \left(\frac{C_2}{r}+ \frac{\beta}{r^{v+1}}\left( \frac{C_4}{\sqrt{r}} + C_5 \right)
\right)$.
Substituting the bounds of $|\phi_t^\beta|$ and $|(\phi_t^\beta)^{\prime}(r)|$ in \eqref{Inq-1}, we have 
\begin{align*}
   & \left| \int_{B[o,\rho]} g(x)\phi_t^\beta(|x|) w_k(x)dx
    \right| \\
    & \leq \mathcal{N}_{\eta}(x_0)t \left( \rho^{n+2\gamma} \eta(\rho) C_1 
    +\int_0^\rho r^{n+2\gamma}\eta(r) 
     \left(\frac{C_2}{r}+ \frac{\beta}{r^{v+1}}\left(\frac{C_4}{\sqrt{r}} + C_5 \right)
\right)
   dr \right)\\
   &= C_6t.
\end{align*}
Using the properties of $\eta$ function, we have  $C_6t \leq C_7\eta(t) = C_7\eta(t^{1-1/\beta}t^{1/\beta}) \leq C_7(1+t^{1-1/\beta})\eta(t^{1/\beta}) \leq C_8\eta(t^{1/\beta})$ for $\beta\geq 1$. Since $t\in (0,1],$ for $0<\beta \leq 1$ we have $C_6 t\leq C_6 t^\beta \leq C_9\eta(t^\beta)$.
\end{proof}
The following result discusses the rate of convergence of the truncated functions $\mathfrak{T}_{\epsilon}(\mathcal{I}_k^\alpha (f))$ at the point $x_0$, where $f$ has the $\eta$-smoothness. We replace the operator with a suitable constant for the convenience of the proof i.e., $\mathfrak{T}_{\epsilon}(f)= \frac{1}{C(\frac{\alpha}{\beta},\nu)}\mathfrak{T}_{\epsilon}(f)$, where   $C(\frac{\alpha}{\beta},\nu)$ is given in Theorem \ref{Inversion-Riesz}.

\begin{theorem}\label{Riez-R-of-C}
Let \( f \in L_k^p(\mathbb{R}^n) \cap L_k^2(\mathbb{R}^n) \) for \( 1 \leq p < \infty \), and suppose that \( f \) possesses \(\eta\)-smoothness at the point \( x_0 \in \mathbb{R}^n \). Then the following estimate holds:
\begin{align*}
    \left| \mathfrak{T}_\epsilon(\mathcal{I}_k^{\alpha}(f))(x_0)- f(x_0) \right| \leq C\,\eta(Y(\epsilon)) \quad \text{as } \epsilon \to 0^+,
\end{align*}
where \( 0 < \alpha < \frac{n+2\gamma}{p} \), \( C>0\) is a constant independent of \( \epsilon\), and the function \( Y(\epsilon) \) is defined by
\begin{equation*}
     Y(\epsilon) =   
\begin{cases} 
  \epsilon^{1/\beta}, & \text{if } \beta \geq 1, \\ 
  \epsilon^{\beta},   & \text{if } 0 < \beta \leq 1.
\end{cases}
\end{equation*}
\end{theorem}
\begin{proof}
We begin the proof by considering  $\mathfrak{T}_\epsilon(\mathcal{I}_k^{\alpha}(f))(x_0)- f(x_0)$. Invoking Lemma \ref{Lemma-R-2}, we have 
\begin{align} \label{m-eq0}
  i_1=  \left| \mathfrak{T}_\epsilon(\mathcal{I}_k^{\alpha}(f))(x_0)- f(x_0) \right| \leq \frac{1}{\left|C(\frac{\alpha}{\beta},\nu)\right|}\int_0^\infty \left|\mathfrak{B}_k^{(\beta, \epsilon s)}f(x_0)-f(x_0)\right| \left|\mathcal{K}_{\alpha/\beta}(s)\right|ds .
\end{align}
    Let us denote $i_2=| \mathfrak{B}_k^{(\beta, \epsilon s)}f(x_0)-f(x_0)|$. Using \eqref{beta-convolution}, we obtain
    \begin{align*}
        i_2 &= \left|\int_{\mathbb{R}^n} \mathcal{W}_k^{(\beta, \epsilon s)}(y) \left( \tau_yf(x_0)-f(x_0) \right)w_k(y)dy
        \right|\\
        &\leq \left|\int_{B[0,\rho]} \mathcal{W}_k^{(\beta, \epsilon s)} (y)\left( \tau_yf(x_0)-f(x_0) \right)w_k(y)dy
        \right| + \\ & \hspace{4.5cm}  \left|\int_{B[0,\rho]^c} \mathcal{W}_k^{(\beta, \epsilon s)}(y) \left( \tau_yf(x_0)-f(x_0) \right)w_k(y)dy
        \right|\\
        & = i_3 +i_4.
    \end{align*}
    In order to estimate $i_3$, we use Corollary \ref{corollary-5.4}. 
    \begin{align*}
      i_3=  \int_{B[0,\rho]} \mathcal{W}_k^{(\beta, \epsilon s)}(y) \left| \tau_yf(x_0)-f(x_0) \right|w_k(y)dy \leq C_1\eta((\epsilon s)^{1/\beta}) \leq C_1 (1+s^{1/\beta}) \eta (\epsilon^{1/\beta}),
    \end{align*} where $ 1\leq \beta \leq 2$,
    and $ i_3 \leq  C^\prime (1+s^{\beta}) \eta (\epsilon^{\beta}), \quad \text{ for $0< \beta \leq 1$. }$\\
\noindent  For $\beta > 2$, invoking \eqref{smoothness-1} and the bound of the function $ \mathcal{W}_k^{(\beta, \epsilon s)}(y), $ we have 
    \begin{align*}
        i_3 & \leq \int_{B[0,\rho]} \left|\mathcal{W}_k^{(\beta, \epsilon s)}(y) \left( \tau_yf(x_0)-f(x_0) \right) \right|w_k(y)dy 
         \leq C_2\epsilon s \mathcal{N}_{\eta}f(x_0) \rho^{n+2\gamma}\eta(\rho)\\ &\leq C_3s\eta(\epsilon^{1/\beta}), \quad \text{where} \quad C_3= \frac{1}{a} \mathcal{N}_\eta f(x_0)\rho^{n+2\gamma}\eta(\rho)C_2.
    \end{align*}

To calculate $i_4$, we decompose the integrals into two parts as
\begin{align*}
    i_4 &\leq \left|\int_{B[0,\rho]^c} \mathcal{W}_k^{(\beta, \epsilon s)}(y)  \tau_yf(x_0) w_k(y)dy
        \right| + \left|\int_{B[0,\rho]^c} \mathcal{W}_k^{(\beta, \epsilon s)}(y)  f(x_0) w_k(y)dy 
        \right| \\
        &= i_5+i_6.
\end{align*}
In view of the properties of convolution and translation, $i_5$ becomes
\begin{align*}
    i_5 &= \left|\int_{B[0,\rho]^c} \mathcal{W}_k^{(\beta, \epsilon s)}(y)  \tau_yf(x_0) w_k(y)dy
        \right| = \left| \int_{\mathbb{R}^n} \chi_{B[0,\rho]^c} \mathcal{W}_k^{(\beta, \epsilon s)}(y)\tau_{x_0}( f)(y)w_k(y)dy \right| \notag \\
        & \leq  \int_{\mathbb{R}^n} \left| \mathcal{W}_k^{(\beta, \epsilon s)}\right|(y)|\tau_{x_0}(f)(y)|w_k(y)dy. \notag
\end{align*} 
Using the Cauchy-Schwarz inequality, we obtain  
\begin{align*}
    i_5 \leq  \|\mathcal{W}_k^{(\beta, \epsilon s)}\|_{L_k^2(\mathbb{R}^n)} \|f\|_{L_k^2(\mathbb{R}^n)} \leq \epsilon s C_4 \leq C_5 s\eta(\epsilon^{1/\beta}), \quad \text{ for }  \beta \geq 1. 
\end{align*}
In addition, 
\begin{align*}
    i_5 \leq \epsilon s C_4 \leq C^{\prime\prime} s\eta(\epsilon^{\beta}), \quad \text{ for } 0< \beta \leq 1.
\end{align*}

\noindent Moreover 
\begin{align*}
    i_6 \leq C_6 s\eta(\epsilon^{1/\beta}) \quad \text{ for $\beta \geq 1$  and } i_6 \leq C^{\prime\prime\prime} s\eta(\epsilon^{\beta}) \quad \text{ for $0<\beta \leq 1$.}
\end{align*}

\noindent We finally define 
 $C^1= \max\{C_1, C_3, C_6 \}$, and $C^*=\max \{C^\prime, C^{\prime\prime},C^{\prime\prime\prime} \}$. Then
\begin{align} \label{m-eq}
    i_2 \leq C^1\eta(\epsilon^{1/\beta})\left( 1+2s+s^{1/\beta} \right) \quad \text{for } \beta\geq 1,
\end{align} and 
\begin{align} \label{m-eq-2}
    i_2 \leq C^*\eta(\epsilon^{\beta})\left( 1+2s+s^{\beta} \right) \quad \text{for }0< \beta\leq 1.
\end{align}
 Invoking \eqref{m-eq} and \eqref{m-eq-2} in \eqref{m-eq0},  we have
\begin{align*}
     i_1  \leq C^1\eta(\epsilon^{1/\beta}) \int_0^\infty  \left( 1+2s+s^{1/\beta} \right)\left|\mathcal{K}_{\alpha/\beta}(s)\right|ds \leq C_1(\nu)\eta(\epsilon^{1/\beta}) \text{  for } \beta\geq 1,
\end{align*} 
and 
\begin{align*}
     i_1 \leq C^*\eta(\epsilon^{\beta}) \int_0^\infty  \left( 1+2s+s^{\beta} \right)\left|\mathcal{K}_{\alpha/\beta}(s)\right|ds \leq C_2(\nu)\eta(\epsilon^{\beta}) \text{ 
 for } 0<\beta\leq 1,
\end{align*}
where the properties of the wavelet measure $\nu$ guarantee the convergence of the integrals. Thus, we have the desired result by choosing constant \( C =\max\{C_1(\nu),C_2(\nu)\}\).
\end{proof}

The following theorem discusses the rate of convergence of truncated hypersingular integrals in the inversion of bi-parametric potential. For the convenience of the proof, we modify the operator $\mathcal{V}_\epsilon$ by $\frac{1}{C\left(\frac{\alpha}{\beta},\nu\right)}\mathcal{V}_\epsilon$.

\begin{theorem}\label{Rate-Pot}

Let $f \in L_k^p(\mathbb{R}^n)\cap L_k^2(\mathbb{R}^n)$ for $1\leq p \leq \infty$, and suppose that $f$ exhibits $\eta$-smoothness at a point $x_0$. Then the following pointwise estimate holds:
  \begin{align*}
      \left| \mathcal{V}_\epsilon\left( \mathfrak{S}_k^{(\alpha, \beta)}(f) \right)(x_0)-f(x_0) \right| \leq C\eta(Y(\epsilon)), \quad \text{ as $\epsilon \longrightarrow 0^+$}, 
  \end{align*}
  where $\alpha, \beta >0$, $C>0$ is a constant independent of $\epsilon$, and the function \(Y(\epsilon)\) is as defined in Theorem \ref{Riez-R-of-C}.
 \end{theorem}
\begin{proof}
    We begin the proof by considering $j_0 = \mathcal{V}_\epsilon\left( \mathfrak{S}_k^{(\alpha, \beta)}(f) \right)(x_0)-f(x_0)$
    Using Lemma \ref{Lemma-P-2}, we have
    \begin{align} \label{Last}
        j_0  &  \leq \frac{1}{\left| C\left(\frac{\alpha}{\beta},\nu\right) \right|} \int_0^\infty \left|e^{-\epsilon s}\mathfrak{B}_k^{(\beta, \epsilon s)}(f)(x_0)-f(x_0)\right| \left|\mathcal{K}_{\alpha/\beta}(s)\right|ds.
    \end{align}Let $j_1 = \left|e^{-\epsilon s}\mathfrak{B}_k^{(\beta, \epsilon s)}(f)(x_0)-f(x_0)\right|$. Then 
    \begin{align*}
        j_1 \leq (1 - e^{-\epsilon s}) 
        \left|\mathfrak{B}_k^{(\beta, \epsilon s)}(f) (x_0)\right|+
        \left|\mathfrak{B}_k^{(\beta, \epsilon s)}(f) (x_0) - f(x_0)\right| = j_2+j_3. \end{align*}
 Using the fact that $(1 - e^{-\epsilon s})\leq C_1s\epsilon$ and   \eqref{Pro.Main-6}, 
it follows that  $j_2 \leq C_2 s\eta(\epsilon^{1/\beta})$ for $\beta \geq 1$ and $j_2 \leq C_3 s\eta(\epsilon^{\beta})$ for $0<\beta \leq 1$.  Now, we utilize the estimate \eqref{m-eq} and \eqref{m-eq-2} for $j_3$. 
Thus, \eqref{Last} can be bounded as follows
\begin{align*}
    j_0 \leq  C_4 \eta(\epsilon^{1/\beta}) \int_0^\infty \left( 1+2s+s^{1/\beta} \right)  \left|\mathcal{K}_{\alpha/\beta}(s)\right|ds \leq C^\prime\eta(\epsilon^{1/\beta}), \text{ for $\beta \geq 1$}
\end{align*}
and 
\begin{align*}
    j_0 \leq  C_5 \eta(\epsilon^{\beta}) \int_0^\infty \left( 1+2s+s^{\beta} \right)  \left|\mathcal{K}_{\alpha/\beta}(s)\right|ds \leq C^{\prime \prime}\eta(\epsilon^{\beta}), \text{ for $0 < \beta \leq 1$}.
\end{align*}
The thesis follows directly from the properties of the wavelet measure and the required constant \( C= \max\{C^\prime, C^{\prime \prime} \}\).
\end{proof}

\noindent\textbf{Acknowledgment:}
The authors express their gratitude for the financial support provided by the Anusandhan National Research Foundation (ANRF) with the reference number SUR/2022/005678, for the first author (SKV) and the UGC (221610147795) for the second author (AP) and their funding and resources were instrumental in the completion of this work. \\  
\\
\textbf{Declarations:}\\
\textbf{Conflict of interest:}
We would like to declare that we do not have any conflict of interest.\\ \\
\textbf{Data availability:}
No data sets were generated or analyzed during the current study.
\bibliographystyle{abbrvnat} 
 
\end{document}